%% file: Xns27arXiv.tex
\documentclass[a4paper,12pt]{amsart}

\usepackage[dvipsnames]{xcolor}
\definecolor{brightcerulean}{rgb}{0.11, 0.67, 0.84}
\usepackage[colorlinks=false, allcolors=brightcerulean, allbordercolors=brightcerulean,
pdfauthor={J. S. Balakrishnan, L. A. Betts, D. R. Hast, A. Jha, J. S. M\"uller},
pdftitle={Rational points on the non-split Cartan modular curve of level 27 and quadratic Chabauty over number fields}
]{hyperref}

\usepackage{url,bbold,dsfont}

\usepackage{tikz, tikz-cd, enumerate,caption,stmaryrd,amsfonts,amsmath, amssymb,systeme,comment,graphicx,colonequals,faktor,xfrac, mathtools}
\usetikzlibrary{matrix,positioning,arrows,shapes,chains,calc,automata}
\usetikzlibrary{decorations.pathreplacing}
\usepackage[shortlabels]{enumitem}
\usepackage[T1]{fontenc}
\usepackage[utf8]{inputenc}
\usepackage[english]{babel}
\usepackage{csquotes}
\usepackage{microtype}

\newcommand\puteqnum{
  \refstepcounter{equation}\textup{(\theequation)}}
\usepackage{comment}
\usepackage{graphicx}
\usepackage{epsfig}
\usepackage{fancyhdr}
\usepackage{mathrsfs}
\usepackage{libertine}

\usepackage[
  backend=biber,
  style=alphabetic, 
  sorting=nyt, 
  backref=true,
  giveninits=true, 
  maxbibnames=99 
]{biblatex}
\renewbibmacro{in:}{}
\DefineBibliographyStrings{english}{%
  backrefpage = {$\uparrow$},%
  backrefpages = {$\uparrow$}%
}
\DeclareSourcemap{%
  \maps[datatype=bibtex]{
     \map{%
        \step[fieldsource=doi,final]%
        \step[fieldset=isbn,null]%
        \step[fieldset=issn,null]%
        }%
      }%
}

\usepackage[margin=1in]{geometry}

\newcommand\cyr{%
  \renewcommand\rmdefault{cmr}%
  \renewcommand\sfdefault{wncyss}%
  \renewcommand\encodingdefault{OT2}%
  \normalfont\selectfont}
\DeclareTextFontCommand{\textcyr}{\cyr}

\definecolor{red}{rgb}{0.9,0,0}
\definecolor{purple}{rgb}{0.8,0,0.6}

\numberwithin{equation}{section}

\DeclareFontFamily{U}{wncy}{}
\DeclareFontShape{U}{wncy}{m}{n}{<->wncyr10}{}
\DeclareSymbolFont{mcy}{U}{wncy}{m}{n}
\DeclareMathSymbol{\Sha}{\mathord}{mcy}{"58}
\DeclareMathOperator{\NS}{\mathrm{NS}}
\newcommand{\1}{{\rm 1\hspace*{-0.4ex}\rule{0.1ex}{1.52ex}\hspace*{0.2ex}}}

\DeclareMathOperator{\Adeles}{\mathbf{A}}
\DeclareMathOperator{\Q}{\mathbf{Q}}
\DeclareMathOperator{\F}{\mathbf{F}}
\DeclareMathOperator{\Z}{\mathbf{Z}}
\DeclareMathOperator{\C}{\mathbf{C}}

\DeclareMathOperator{\cX}{\mathcal{X}}
\DeclareMathOperator{\cY}{\mathcal{Y}}

\DeclareMathOperator{\HH}{\mathrm{H}}
\DeclareMathOperator{\cO}{\mathcal{O}}
\DeclareMathOperator{\ord}{\mathrm{ord}}
\DeclareMathOperator{\rk}{\mathrm{rk}}
\DeclareMathOperator{\SL}{SL}

\DeclareMathOperator{\dR}{\scriptstyle \mathrm{dR}}
\DeclareMathOperator{\an}{\scriptstyle \mathrm{an}}

\DeclareMathOperator{\fil}{\scriptstyle \mathrm{ Fil}}
\DeclareMathOperator{\cris}{\scriptstyle \mathrm{ cris}}
\DeclareMathOperator{\ns}{\scriptstyle \mathrm{ ns}}

\DeclareMathOperator{\rig}{\scriptstyle \mathrm{{rig}}}

\DeclareMathOperator{\Aut}{\mathrm{Aut}}

\DeclareMathOperator{\End}{\mathrm{End}}

\DeclareMathOperator{\Gal}{\mathrm{Gal}}
\DeclareMathOperator{\Res}{\mathrm{Res}}

\DeclareMathOperator{\Fil}{\mathrm{Fil}}

\DeclareMathOperator{\Tr}{\mathrm{Tr}}

\DeclareMathOperator{\Jac}{\mathrm{Jac}}

\DeclareMathOperator{\GL}{\mathrm{GL}}

\newcommand{\fp}{\mathfrak{p}}
\newcommand{\spl}{\mathrm{spl}}

\newcommand{\resdisc}[1]{\left]#1\right[}
\newcommand{\et}{\mathrm{\acute{e}t}}

\newcommand\into{\hookrightarrow}

\newtheorem{theorem}{Theorem}[section]
\newtheorem{prop}[theorem]{Proposition}

\newtheorem{lemma}[theorem]{Lemma}

\theoremstyle{definition}

\newtheorem{algo}[theorem]{Algorithm}
\theoremstyle{remark}
\newtheorem{Remark}[theorem]{Remark}

\usepackage[all,cmtip]{xy}
\input{abbreviations_xns.tex}

\begin{document}

\title[Rational points on $X_{\ns}^+(27)$ and quadratic Chabauty over number fields]{Rational points on the non-split Cartan modular curve of level 27 and quadratic Chabauty over number fields}

\author[J. S. Balakrishnan]{Jennifer S. Balakrishnan}
\author[L. A. Betts]{L.\ Alexander Betts}
\author[D. R. Hast]{Daniel Rayor Hast}
\author[A. Jha]{Aashraya Jha}
\author[J. S. M\"uller]{J. Steffen M\"uller}


\begin{abstract}
Thanks to work of Rouse, Sutherland, and Zureick-Brown, it is known exactly
  which subgroups of~$\GL_2(\Z_3)$ can occur as the image of the $3$-adic
  Galois representation attached to a non-CM elliptic curve over~$\Q$, with
  a single exception: the normaliser of the non-split Cartan subgroup of
  level~$27$. In this paper, we complete the classification of $3$-adic
  Galois images by showing that the normaliser of the non-split Cartan subgroup of level~$27$ cannot occur as a $3$-adic Galois image of a non-CM elliptic curve.

Our proof proceeds via computing the $\Q(\zeta_3)$-rational points on a
  certain smooth plane quartic curve $X'_H$ (arising as a quotient of the
  modular curve~$X_{\ns}^+(27)$) defined over~$\Q(\zeta_3)$ whose Jacobian
  has Mordell--Weil rank~6.  To this end, we describe how to carry out the quadratic Chabauty method for a modular curve $X$ defined over a number field $F$, which, when applicable, determines 
a finite subset of $X(F\otimes\Q_p)$ in certain situations of larger
  Mordell--Weil rank than previously considered. Together with an analysis of local heights above 3, we apply this quadratic Chabauty method to determine $X'_H(\Q(\zeta_3))$. This allows us to compute the set $X_{\ns}^+(27)(\Q)$, finishing the classification of 3-adic images of Galois.
\end{abstract}
\maketitle



\section{Introduction}
After his landmark work classifying rational points on the modular curves $X_0(\ell)$  and $X_1(\ell)$ for prime level $\ell$ \cite{mazur1977modular, mazur1978}, Mazur outlined a ``Program B'' \cite{mazurprogramB}  to study rational
points on all modular curves associated to subgroups of
$\GL_2(\widehat{\Z})$. The case of prime power level, or the study of
$\ell$-adic images of Galois representations of elliptic curves over $\Q$,
has seen much progress in recent years.   Rouse and Zureick-Brown
\cite{dzbrouse2adic} classified all 2-adic images.  Combined with
results due to various authors (see~\cite[Section~4]{rszbadic} for an
overview), work by Balakrishnan,
Dogra, M\"uller, Tuitman, and Vonk on quadratic Chabauty for  the split
Cartan modular curve $X_{\textrm{s}}^+(13)$ \cite{BDMTV_split}, the
$S_4$-image curve $X_{S_4}(13)$, and the non-split Cartan modular curve
$X_{\ns}^+(17)$ \cite{balakrishnan_dogra_müller_tuitman_vonk_2023} finished
the classification of 13-adic and 17-adic images. Rouse, Sutherland, and Zureick-Brown computed rational points on all remaining modular curves, aside from those dominating two modular curves of level 49 and genus 9 and the non-split Cartan curves of level 27, 25, 49, 121, and prime level greater than 17  \cite{rszbadic}. 

The non-split Cartan modular curve $X_{\ns}^+(27)$ has genus 12 and
Mordell--Weil rank 12, and determining its set of rational points would finish the classification of 3-adic images of Galois. However, the large genus and rank preclude existing computational techniques from being applicable. In \cite[\S 9.1]{rszbadic}, Rouse, Sutherland, and Zureick-Brown  describe various computations they carried out with $X_{\ns}^+(27)$, including exhibiting a smooth plane quartic curve $X'_H$ over $F = \Q(\zeta_3)$ together with a degree 3 morphism $X_{\ns}^+(27) \rightarrow X'_H$ defined over $F$. Thus every $F$-rational point of $X_{\ns}^+(27)$ maps to a $F$-rational point of $X'_H$, and determining $X'_H(F)$ would readily yield $X_{\ns}^+(27)(F)$, and consequently $X_{\ns}^+(27)(\Q)$.

The genus 3 curve $X'_H/F$ has Jacobian rank 6 and is not within the scope
of quadratic Chabauty as in
\cite{balakrishnan_dogra_müller_tuitman_vonk_2023, QCMod} for two reasons:
it is not defined over $\Q$, and its Jacobian rank is too large. 
 Nevertheless, combining quadratic Chabauty with restriction of scalars
allows one, in principle, to consider curves over number fields with larger Jacobian rank.
This was demonstrated in \cite{BBBMnumberfields} using Coleman--Gross
$p$-adic heights for certain hyperelliptic curves.  

In the present work, we develop a quadratic Chabauty method that uses
restriction of scalars and  Nekov\'a\v{r} $p$-adic heights constructed with
multiple independent id\`ele class characters, that builds on the
approaches of \cite{balakrishnan_dogra_müller_tuitman_vonk_2023,
BBBMnumberfields} to study rational points on curves (with larger Jacobian
rank) defined over number fields beyond $\Q$. We apply an implementation of
this method in the computer algebra system {\tt Magma}~\cite{MR1484478} to the curve
$X_H'$ and combine it with a computation of the local heights above $3$ to prove the following: 

\begin{theorem}\label{thm:XH'F} $\#X_H'(\Q(\zeta_3)) = 13$.\end{theorem}

As discussed in~\S\ref{S:proofs}, this allows us to deduce the following:

\begin{theorem}\label{thm:ns27F} We have $\#X_{\ns}^+(27)(\Q(\zeta_3)) =
10$, and all of these points are CM. There are $3$ points with discriminant $-4$, and the remaining $7$ points have discriminants $-7,-16,-19,-28,-43,$ $-67,$ and $-163.$
\end{theorem}

Examining the points in Theorem~\ref{thm:ns27F}, we immediately find:

\begin{theorem}\label{thm:ns27} We have $\#X_{\ns}^+(27)(\Q) = 8$, and all
points are CM,  with respective discriminants $-4,-7,-16,-19,-28,-43,-67,-163$.\end{theorem}

Theorem \ref{thm:ns27} finishes the classification of 3-adic images of Galois representations attached to elliptic curves over $\Q$, after the work of Rouse, Sutherland, and Zureick-Brown:

\begin{theorem}\label{thm:3adic} Let $E/\Q$ be a non-CM elliptic curve.
Then  $\textrm{im}\, \rho_{E,3^\infty}$ is one of 47 subgroups of
$\GL_2(\Z_3)$ of level at most 27 and index at most 72, as in \cite[Table
3]{rszbadic}.\end{theorem}

Moreover, Theorem \ref{thm:ns27} gives a new proof of the class number one problem, via the determination of rational points on a non-split Cartan modular curve, as noted by Serre \cite[Appendix]{serre}.

In Section 2, we give an overview of the quadratic Chabauty method over number fields, as enriched by a restriction-of-scalars construction. In particular, we describe how to compute $\End^0(J)$-equivariant Nekov\'a\v{r}  $p$-adic heights and incorporate the use of multiple id\`ele class characters, producing a system of multivariate $p$-adic power series. In Section 3, we discuss how to apply this to a particular smooth plane quartic quotient of $X_{\ns}^+(27)$ that is defined over $\Q(\zeta_3)$  and use this to determine $X_{\ns}^+(27)(\Q)$.

\subsection*{Acknowledgements} We are very grateful to Ole Ossen, Jeremy
Rouse, Andrew Sutherland, and David Zureick-Brown for numerous conversations throughout the course of this project, and to the Simons Collaboration in Arithmetic Geometry, Number Theory, and Computation for facilitating this work. 
Some of the results here were based on material in the fourth author's Ph.D. thesis, and it is a pleasure to thank his dissertation committee, including Bjorn Poonen, David Rohrlich, and Padmavathi Srinivasan. 
We would also like to thank
Eran Assaf, Francesca Bianchi, Edgar Costa, Henri Darmon, Netan Dogra, Juanita Duque-Rosero, Timo Keller, Martin L\"udtke, Jeroen Sijsling, John Voight, Jan Vonk, and Stefan Wewers for helpful discussions.

JSB was partially supported by  NSF grant DMS-1945452 and Simons Foundation
grant nos.\ 550023 and 1036361, LAB was supported by Simons Foundation grant no.\ 550031, DRH and AJ by Simons Foundation
grant no.\ 550023, and JSM by NWO grant VI.Vidi.192.106.

\section{Quadratic Chabauty over number fields} \label{Section : QC for modular curves}
In this section, we describe the theoretical and computational aspects of quadratic Chabauty over number fields. One key idea is combining quadratic Chabauty (in particular, the formulation of quadratic Chabauty with Nekov\'a\v{r} heights) with restriction of scalars.

\subsection{Notation}

We fix the following notation, once and for all.
\begin{itemize}
    \item $F$: A number field of degree $d$. 
    \item $O_F$: The ring of integers of $F$.
    \item  $G_F$: The absolute Galois group of $F$.
    \item $G_v$: The absolute Galois group of $F_v$ for a finite place $v$ of
 $F$. 
\item $p$: A prime of $\Q$ that splits completely in $F$.
\item $\sigma_\fp\colon F \hookrightarrow F_{\fp}$: The
  completion of $F$ at a prime $\fp$ of $F$.
 \item $X/F$: a smooth, projective, geometrically irreducible curve of
   genus $g>1$ such that $X(F)\ne \emptyset$ and such that $X$ has good reduction at
    all places above $p$.
\item $\sigma\colon X(F)
  \hookrightarrow X(F\otimes \Q_p) = \prod_{\fp\mid p} X_{\fp}(F_\fp)$: the
    embedding induced by the  $\sigma_\fp$.
 \item $J=\Jac(X)$: The Jacobian of $X$.
 \item $r=\rk(J(F))$: The Mordell--Weil rank of~$J/F$.
 \item $b\in X(F)$: A choice of base point.
\item  $\iota=\iota_b\colon X\into J$: The Abel--Jacobi map with respect to
  $b$.
   \item $T_0\coloneqq H^0(X,\Omega^1)^*\simeq H^0(J,\Omega^1)^*$.
 \item $T_\fp\coloneqq
H^0(X_{\fp},\Omega^1)^*=T_0\otimes F_{\fp}$.
\item $T\coloneqq T_0\otimes\Q_p\simeq \bigoplus_{\fp \mid p} T_\fp$.
\item $K=\End^0_F(J)\coloneqq \End_F(J)\otimes\Q$: The
  endomorphism algebra of $J$.
\item $r_{\NS}$: The rank of the N\'eron--Severi group $\NS(J)$.
 \item $V=H^1_{\et}(X_{\bar{F}},\Q_p)^*$: The $p$-adic Tate module of $X$.
\end{itemize}

\subsection{Linear Chabauty and restriction of scalars}\label{subsec:Chabauty}

We begin by recalling how the construction goes for the classical Chabauty--Coleman method.
Chabauty \cite{chabauty1941points} showed that the set $X(F)$ is finite if $r<g$. Over~$F=\Q$, the idea underpinning Chabauty's method is to intersect $X(\Q_p)$ with the $p$-adic closure~$\overline{J(\Q)}$ of the Mordell--Weil group of the Jacobian, viewing both as subspaces of~$J(\Q_p)$ via the chosen Abel--Jacobi map~$\iota$. The rational points are contained in the intersection $X(\Q_p)\cap\overline{J(\Q)}$, and a dimension count suggests that the intersection $X(\Q_p)\cap\overline{J(\Q)}$, and hence the set of rational points~$X(\Q)$, should be finite when ~$1 + r \leq g$. This turns out to be correct, and Coleman showed in~\cite{coleman1985effective} how to turn this idea into a method for bounding $\#X(\Q)$ or computing $X(\Q)$.

One would like to use a similar idea for other number fields $F$, using all
primes $\fp\mid p$ simultaneously.
Based on
an idea of Wetherell, Siksek \cite{siksek} (see
also~\cite{triantafillou2021restriction}) looked at the Weil restrictions of scalars $V\coloneqq \Res^F_{\Q}X$ and $A\coloneqq \Res^F_{\Q}J$. He exploited the fact that $X(F)=V(\Q)$ and that
\begin{equation*}\label{eqn:Weil restriction}
    V(\Q)\subseteq V(\Q_p)\cap \bar{A(\Q)}\subseteq A(\Q_p)\,,
\end{equation*}
where $V(\Q_p)$ has dimension $d$ and $A(\Q_p)$ has dimension $dg$.
A similar heuristic as above predicts that the intersection $V(\Q_p)\cap\overline{A(\Q)}$ should be finite as soon as $r\leq d(g-1)$.
However, in contrast to the case $F=\Q$, this is not always
true, for instance when $X$ can be defined over $\Q$ and satisfies
$\rk J(\Q) \ge g$, see~\cite[Section~2]{siksek}. For further discussions, including sufficient
conditions for finiteness and a more involved  
counterexample see \cite[\S2.2]{dogra2023unlikely}; see
also~\cite[Section~4]{Has21} and~\cite{triantafillou2021restriction}.

Nonetheless, one can still try to use this idea to compute $F$-rational points in practice, even if the resulting method may not work in every case. To do so, one would like to construct explicit relations in the image of the abelian
logarithm
\begin{equation}\label{log2}
  \log\colon J(F)\otimes \Q_p\to \Res^F_{\Q}(J)(\Q_p)\to
  \Lie(\Res^F_{\Q}(J))_{\Q_p}.
\end{equation}
More concretely, following \cite[Section 3]{BBBMnumberfields}, consider the functions 
\begin{equation}\label{fscoordfree}
 \int_b \omega
  \colon X(F_{\fp})\to F_\fp=\Q_p\,,\quad \text{where}\;\fp\mid p\,,\; \omega \in
  H^0(X_{F_{\fp}},\Omega^1)\,.
  \end{equation}
 
As~$\fp$ and~$\omega$ vary, one obtains~$dg$ functions
\[
X(F\otimes\Q_p)=\prod_{\fp\mid p}X(F_{\fp})\to\Q_p
\]
that are $\Q_p$-linearly independent. Each of these extends to a continuous homomorphism
\[
J(F\otimes\Q_p)=\prod_{\fp\mid p}J(F_{\fp})\to\Q_p \,.
\]
Accordingly, we obtain at least $dg-r$   functions that are $\Q_p$-linearly
independent on~$J(F\otimes\Q_p)$ and vanish on~$J(F)$. Restricting to~$X$ gives us $p$-adic locally analytic functions \begin{equation}\label{rhochab}
  \rho_k\colon X(F\otimes\Q_p)\to\Q_p
\end{equation}
for $1\leq k\leq dg-r$, such that $ \rho_k$ vanishes on  $X(F)$. This construction is only useful when~$r\leq d(g-1)$, when it provides us with at least $d$ functions which we might reasonably hope to cut out a finite set.

Siksek discusses an explicit method to compute $X(F)$ using these
functions and applies it to solve certain generalised Fermat equations. For
algorithms to compute the integrals~\eqref{rhochab}, as well as more general Coleman
integrals see~\cite{BBK} and~\cite{balakrishnantuitman}.

\subsection{Summary of quadratic Chabauty over number fields}
\label{Section: QC summary}

The quadratic Chabauty method is an extension of Chabauty's method. The aim of quadratic Chabauty over number fields is to determine $X(F)$ even in some cases when the inequality $r\leq d(g-1)$ fails.
Similar to Chabauty's method, the central idea of quadratic Chabauty is to compute non-constant Coleman
analytic $\Q_p$-valued functions \[\rho\colon X(F\otimes\Q_p)\to\Q_p\] and
finite sets $\Upsilon$ such that 
\[
  \rho(X(F))\subseteq \Upsilon\,.
\]
Following~\cite[Section~1]{BDMTV_split}, we call $(\rho,\Upsilon)$ a \emph{quadratic Chabauty pair}. 
This is similar to linear Chabauty, where we have $\Upsilon =\{0\}$ and $\rho$ is of the form~\eqref{fscoordfree}.

In quadratic Chabauty, $\rho$ and $\Upsilon$ are
determined by $p$-adic heights, as we discuss below.

If $F=\Q$, 
a quadratic Chabauty pair $(\rho, \Upsilon)$ ensures that $\rho^{-1}(\Upsilon)$ is a finite superset of
$X(\Q)$, since 
non-constant Coleman analytic functions on curves have finite fibres.
If we can compute $\rho$ and $\Upsilon$ explicitly, then we can
use this to find the $\Q$-rational points on $X$.

Here we discuss the case of an arbitrary number field $F$.
The linear Chabauty method exploits linear relations in the image of the
abelian log map~\eqref{log2}. 
 In contrast, the goal of quadratic Chabauty is to produce quadratic
relations in the image. 

As above, we can only expect to cut out a finite superset of $X(F)$ in this way if we can find at least $d=\dim_{\Q_p}X(F\otimes\Q_p)$ Chabauty pairs $(\rho_i,\Upsilon_i)$. 
 If this can be done explicitly, then we may compute
the intersection 

\begin{equation}
    \label{Eqn: QC intersections}
    \bigcap _{i=1}^d \rho_i^{-1}(\Upsilon_i)\supseteq
X(F).
\end{equation}

In our case, the quadratic Chabauty pair is obtained from the theory of
Nekov\'a\v{r} heights \cite{nekovar1990p}. We sketch the construction here,
and we give more detail in the rest of this section. 
We make the following strong assumption, since it is satisfied for
our example in Section~\ref{sec3} and simplifies the exposition considerably.
See~\S\ref{subsec:extensions} for generalisations.

\begin{Assump}\label{A:log}
  We shall assume the abelian log map~\eqref{log2} is an isomorphism, and consequently that $r=dg$. We also assume $r_{\NS}\geq 2$.
\end{Assump} 

Recall that we write 
\[
T = \Lie(\Res^{F}_{\Q}(J))_{\Q_p} = H^0(X,\Omega^1)^*\otimes_{\Q}\Q_p
\]
for the target of the abelian logarithm map, which is the product over $\fp \mid p$ of the vector spaces
\[
T_{\fp}= \Lie(\Res^{F_{\fp}}_{\Q_p}(J_{F_{\fp}})) = H^0(X_{F_{\fp}},\Omega^1)^*\,.
\]

Given a continuous {id\`ele class character}
$\chi\colon\Adeles_F^\times/F^\times \to \Q_p$ 
and, for each $\fp\mid p$, a splitting $s_\fp$ of the Hodge filtration
$T_\fp^*\subseteq H^1_{\dR}(X_{F_{\fp}})$, Nekov\'a\v{r} defines a global bilinear $p$-adic height pairing
\[
h=h^\chi\colon (J(F)\otimes\Q_p)\times(J(F)\otimes\Q_p) \to \Q_p
\]
(see~\S\ref{subsec:Nekovar}). One can turn this pairing into a function on~$X(F)$ by choosing a suitable cycle $Z$ of 
$X\times X$. If~$E_Z\in\End(J)$ denotes the corresponding endomorphism, then we denote by $h_Z\colon X(F) \to \Q_p$ the function given by
\[
h_Z(x) = h([x-b],E_Z([x-b])+c_Z) \,,
\]
with $c_Z\in J(F)$ the Chow--Heegner point associated to~$Z$; see~\cite[Remarks~3.11
and~5.6]{BDMTV_split}.

\begin{Remark}
    We will freely use the identifications
    \[
    \NS(J) = \End^+(J) = \mathrm{Corr}^+(X,X)
    \]
    between the N\'eron--Severi group~$\NS(J)$ of~$J$, the
    group~$\End^+(J)$ of Rosati-symmetric endomorphisms of~$J$,  and the
    group~$\mathrm{Corr}^+(X,X)$ of correspondences from~$X$ to~$X$ that
    are invariant under transposition \cite[Proposition~5.2.1(b) \&
    Theorem~11.5.1 \&
    Proposition~11.5.3]{birkenhake-lange:complex_abelian_varieties}. That
    is, we will want to think of elements of~$\NS(J)$ sometimes as
    endomorphisms, and sometimes as correspondences, as needed. We will
    also tacitly permit ourselves to represent elements
    of~$\mathrm{Corr}(X,X)$ by divisors on~$X\times X$. By the discussion
    in~\cite[\S2.1]{DLF19}, an element of $\End^+(J)$ that has trace zero corresponds to an element $\ker(\NS(J)\to\NS(X))$, and such an element was referred to as a \emph{nice correspondence} in \cite{BDMTV_split}. 
\end{Remark}

Balakrishnan and Dogra show that the height function~$h_Z$ admits a local decomposition 
\[
h_Z(x) = \sum_vh_{Z,v}(x) \,,
\]
where the sum is taken over finite places~$v$ of~$F$ and
\[
h_{Z,v}\colon X(F_v) \to \Q_p
\]
is the so-called \emph{local height} function. For~$v\nmid p$, the image~$\Upsilon_v$ of~$h_{Z,v}$ is finite and equals $\{0\}$ for all but finitely
many $v$ \cite[Lemma~5.4]{balakrishnandograQC1}. Consequently, the function
\[
\rho\coloneqq h_Z-\sum_{\fp\mid p}h_{Z,\fp}
\]
takes values in the finite set
\begin{equation}\label{Upsilon}
\Upsilon\coloneqq \left\{\sum_v y_v \,:\, y_v\in \Upsilon_v\right\}
\end{equation}
when evaluated on~$x\in X(F)$.

To produce a quadratic Chabauty pair $(\rho,\Upsilon)$, we need to
explain how to interpret~$\rho$ as a Coleman function
on~$X(F\otimes\Q_p)=\prod_{\fp\mid p}X(F_{\fp})$. The local
heights~$h_{Z,\fp}$ are already Coleman functions, but the global
height~$h_Z$ is only \emph{a priori} a function on~$X(F)$. To extend~$h_Z$
to~$X(F\otimes\Q_p)$, we use Assumption~\ref{A:log} to reinterpret the
height pairing as a $\Q_p$-bilinear locally analytic pairing~$h\colon T\times T\to\Q_p$. This allows us to extend~$h_Z$ to a Coleman function on~$X(F\otimes\Q_p)$ by
\[
h_Z(x) \coloneqq h(\log([x-b]),\log(E_Z([x-b])+c_Z)) \,.
\]
This shows:

\begin{theorem}\label{T:BalDog}
  The function $\rho$ satisfies the following properties:
  \begin{enumerate}[(a)]
    \item $\rho$ is Coleman analytic,
    \item $\rho(X(F))\subset \Upsilon$, and
    \item $\Upsilon$ is finite.
  \end{enumerate}
\end{theorem}

In other words, $(\rho, \Upsilon)$ defines a quadratic Chabauty pair 
if $\rho$ is non-constant. For $F=\Q$, this is known,
see~\cite[Lemma~3.7]{BDMTV_split}, which summarises results in~\cite{balakrishnandograQC1, balakrishnandograQC2}. 
Consequently, if $F=\Q$, then $\rho^{-1}(\Upsilon)$ is finite.
In general, we also expect $\rho$ to be non-constant, and we can check
this for any given example. Moreover, as discussed in~\S\ref{subsec:Chabauty}, we need at least
$d$ quadratic Chabauty pairs $(\rho_i, \Upsilon_i)$ to have any chance for finiteness. We can obtain different pairs by varying the choices
of~$\chi$ and of~$Z$. 
(As we will discuss in~\S\ref{subsec:solving}, it is advantageous not to also vary the splittings~$s_{\fp}$.)

Since Coleman functions can be expanded into convergent power series on polydiscs, if the set \[\bigcap_i \rho_i^{-1}(\Upsilon_i)\] is finite, we can use Theorem~\ref{T:BalDog} to compute a finite subset of $X(F\otimes \Q_p)$ containing $X(F)$. 

In order to make this result explicit, we have to solve the following algorithmic
problems:
\begin{enumerate}
  \item Compute the action of a nontrivial trace zero cycle~$Z$
    on $H^1_{\dR}(X)$ (see Remark~\ref{R:Z}).
  \item Compute the sets $\Upsilon_v$ for $v\nmid p$ (see~\S\ref{Sec:
    Heights away from p}).
  \item Expand the functions $h_{\fp}=h_{Z,\fp}$ into convergent power
    series on residue polydiscs (see~\S\ref{Sec: htsAbovep} -- \S\ref{subsec:Frob}). 
  \item Expand the function~$h_Z$ into convergent power series on residue polydiscs (see~\S\ref{subsec:solving}). 
 
  \item Find the common roots of a system of $p$-adic power series
    (see~\S\ref{subsec:Hensel}). 
\end{enumerate}
In the remainder of this section, we summarise the theoretical background, and we describe how to solve these problems.

\subsection{Heights on Selmer varieties}\label{subsec:Nekovar}

In this subsection, we review some facts we will need about local and global heights.

\subsubsection{Nekov\'a\v{r}'s heights}\label{subsubsec:Nekovar}
 Nekov\'a\v{r} \cite{nekovar1990p} constructs 
global height pairings
\[H^1_f(G_F, V )\times H^1_f(G_F, V^*(1))\too  \Q_p\]
for reasonably
 behaved $p$-adic representations~$V$ of the absolute Galois group~$G_F$, using $p$-adic Hodge theory. We refer the reader to~\cite[Section 2]{nekovar1990p} 
 for
 the necessary background in $p$-adic Hodge theory and to~\cite[Section
 4]{nekovar1990p}, \cite[Section~3]{BDMTV_split}
 and~\cite[Section~3]{BM_AWS}  for more details of the
 construction. In this subsection, we will only be concerned with the
 $p$-adic Tate module $V=H^1_{\et}(X_{\bar{F}}, \Q_p)^*$. 

The group~$H^1_f(G_F,V)$ (respectively $H^1_f(G_F,V^*(1))$) parametrises extensions of~$\Q_p$ by~$V$ (respectively~$V$ by $\Q_p(1)$) which are crystalline at all places~$\fp\mid p$ and which split $I_v$-equivariantly for all places~$v\nmid p$, where~$I_v$ is the inertia group. If~$e_1\in H^1_f(G_F,V)$ and $e_2\in H^1_f(G_F,V^*(1))$ are classes of extensions~$E_1$ and~$E_2$, respectively, then Nekov\'a\v{r} shows that the cup product~$e_1\cup e_2\in H^2(G_F,\Q_p(1))$ vanishes. In terms of extensions, this means that there exists a representation~$E$ fitting into the commutative diagram
\begin{center}
\begin{equation}\label{diag:mixed}
  \begin{tikzcd}
	&                &                  0 \arrow[d]                        & 0 \arrow[d]           &   \\
	0 \arrow[r] & \Q_p(1) \arrow[d, "="] \arrow[r] & E_2 \arrow[r] \arrow[d]       & V \arrow[r] \arrow[d] & 0 \\
	0 \arrow[r] & \Q_p(1) \arrow[r] & E \arrow[d] \arrow[r]         & E_1 \arrow[r]
          \arrow[d]         &  0 \\
	&                & \Q_p \arrow[d] \arrow[r, "="] & \Q_p \arrow[d]           &   \\
	&                & 0                             & 0                     &  
.
	\end{tikzcd}
\end{equation}
\end{center}
Such an~$E$ is called a \emph{mixed extension} of~$\Q_p$ by $V$ by~$\Q_p(1)$; it is equivalently a $G_F$-representation equipped with a $G_F$-stable filtration and identifications of the graded pieces with~$\Q_p$, $V$, and~$\Q_p(1)$. The set of mixed extensions lifting~$E_1$ and~$E_2$ is a torsor under the group $H^1(G_F,\Q_p(1))$. We write~$\mathrm{MExt}_{G_F,f}(\Q_p,V,\Q_p(1))$ for the set of mixed extensions where~$E_1$ and~$E_2$ satisfy the local conditions mentioned above and write
\[
(\pi_1,\pi_2) \colon \mathrm{MExt}_{G_F,f}(\Q_p,V,\Q_p(1)) \to H^1_f(G_F,V)\times H^1_f(G_F,V)
\]
for the projection sending~$E\mapsto(E_1,E_2)$.

One can consider mixed extensions for local Galois groups $G_v$
and pairs of classes $(e_1,e_2) \in H^1_f(G_v,V)\times H^1_f(G_v,V^*(1))$ in an analogous way. As before, the subscript~$f$ denotes local conditions: if~$v\mid p$, then the extensions should be crystalline; and if~$v\nmid p$ then the extensions should split upon restricting to inertia. (For~$v\nmid p$ a place of good reduction, so the $G_v$-action on~$V$ is unramified, this is equivalent to requiring that the extensions themselves are unramified.) Given a nontrivial continuous homomorphism $\chi_v\colon
F_v^{\times}\to \Q_p$, and if $v\mid p$, also a choice of splitting $s_v$ of the
Hodge filtration $H^0(X_{F_v},\Omega^1)\subseteq
H^1_{\dR}(X_{F_v})$, Nekov\'a\v{r} describes 
a local height function
\[
h_v\coloneqq h_v^{\chi_v}\colon\mathrm{MExt}_{G_v,f}(\Q_p,V,\Q_p(1)) \to \Q_p \,.
\]
We do not discuss Nekov\'a\v{r}'s original construction here; it is described in
detail in~\cite[Section 2]{nekovar1990p},~\cite[Section~3]{BM_AWS}, and ~\cite[\S3.2,
3.3]{BDMTV_split}. Instead, we will focus on explicit formulas,
see~\S\ref{Sec: Heights away from p}--
\S\ref{subsec:Frob}.

In order to combine the local heights into a global height, one chooses the local characters~$\chi_v$ to be the components of a
continuous
{id\`ele class character}  
$\chi\colon\A_F^{\times}/F^{\times}\to \Q_p$.
If $(e_1,e_2)\in H^1_f(G_F,V)\times H^1_f(G_F,V^*(1))$ and $E$ is a
choice of global mixed extension lifting $e_1,e_2$, then Nekov\'a\v{r} shows that 
\[ h(E)\coloneqq h^{\chi}(E) \coloneqq \sum_v h_v(E|_{G_v}) =
\sum_v h^{\chi_v}_v(E|_{G_v})\]
 is independent of the mixed extension and only depends on
 $e_1,e_2$, where $v$ runs through all finite places of $F$. Using Poincar\'e
 duality to identify $V=V^*(1)$, we obtain a function \[h=h^{\chi}\colon H^1_f(G_F, V )\times H^1_f(G_F, V)\too  \Q_p,\]which Nekov\'a\v{r} shows is a bilinear pairing: the \emph{global height pairing} relative to~$\chi$ and the local splittings~$s_\fp$ for~$\fp\mid p$. We remark that the group $J(F)\otimes\Q_p$ embeds inside~$H^1_f(G_F,V)$ via the Kummer map, and so we permit ourselves to conflate the pairing~$h$ with its restriction to~$(J(F)\otimes\Q_p)\times(J(F)\otimes\Q_p)$. (Assuming the Tate--Shafarevich Conjecture, one would even have $H^1_f(G_F,V)=J(F)\otimes\Q_p$, but we don't strictly need this.)

\subsubsection{A mixed extension via twisting}\label{S:HeightsTwisting}

Recall that we are aiming to use the theory of heights and an element $Z\in\ker(\NS(J)\to\NS(X))$ to build quadratic Chabauty pairs~$(\rho,\Upsilon)$. To do this, Balakrishnan and Dogra associate to~$Z$ and the chosen base point~$b$ a certain $G_F$-representation $A\coloneqq A_{Z,b}$, constructed out of the unipotent fundamental group~$U$ of~$X_{\bar{F}}$ based at~$b$. The representation~$A$ is a mixed extension of~$\Q_p$ by~$V$ by~$\Q_p(1)$. For any~$x\in X(F)$, one can twist the representation~$A$ along the unipotent torsor of paths~$P(b,x)$, to obtain another $G_F$-representation
\begin{equation}\label{Eqn: A(x)}
  A(x) \coloneqq A_{Z,b}(x) \coloneqq A\times_U P(b,x)
\end{equation}
which is also a mixed extension of~$\Q_p$ by~$V$ by~$\Q_p(1)$. We refer to~\cite[\S 5.1]{balakrishnandograQC1} and~\cite[\S 3.3]{balakrishnandograQC2} for more details of
the construction and for proofs of the following properties:

\begin{lemma}\label{L:tauprops} 
  The $G_F$-representation $A(x)$ satisfies:
  \begin{enumerate}[(i)]
    \item For all $x\in X(F)$, the representation $A(x)$ is a mixed extension of
      $G_F$-representations with graded pieces $\Q_p, V, \Q_p(1)$, crystalline at all $\fp\mid p$.

    \item Let 
$c_Z\in J$ be the Chow--Heegner point associated to $Z$ (see~\cite[Remarks~3.11
and~5.6]{BDMTV_split}) and let $E_Z\in \End(J)$ be induced by $Z$. 
Then   for all $x\in X(F)$, we have
      \begin{align*}\pi_1(A(x)) &= \log([x-b]),\\
      \pi_2(A(x)) &= 
      \log(E_Z([x-b])+c_Z)\,.\end{align*} 
  \end{enumerate}
\end{lemma}
One can similarly define a $G_v$-representation $A_v(x)$ for $x\in X(F_v)$
and $v$ any finite place of $F$. Lemma~\ref{L:tauprops} extends to
$A_v(x)$; in particular $A_v(x)$ is crystalline if $v\mid p$.

We shall abuse notation, and for any $x\in X(F_v)$ write 
\begin{equation*}
h_v(x)\coloneqq h_v(A_v(x)) 
\end{equation*}
(where we assume that $Z$, $\chi_v$ and, if $v\mid p$, $s_v$
have been fixed). 
We write $h_v^{\chi_v}(x)$ when we want to emphasise the dependence on $\chi_v$.

\begin{Remark}
For any  degree $0$ divisor on $X$, Nekov\'a\v{r}
  \cite[Section 5]{nekovar1990p} describes a class in $H^1_f(G_F,V)$ via an
  \'etale Abel--Jacobi map. Balakrishnan and Dogra construct a divisor $D_{Z,b}(x)$ (see~\cite[Definition
    6.2]{balakrishnandograQC1}) such that the representation $A_{Z,b}(x)$
    in \eqref{Eqn: A(x)} is a mixed extension of the images of the divisors $x-b$ and
    $D_{Z,b}(x)$ under the \'etale Abel--Jacobi map.
\end{Remark}

\subsubsection{Local heights away from $p$}\label{Sec: Heights away from p}
For $v\nmid p$, the only information we need about $h_v$ is that it factors
through Kim's higher Albanese map $j_v$ . By
\cite[Corollary 0.2]{kimtamagawa}, this map has finite image, and this image is~0 for $v$ of
potentially good reduction \cite[Lemma 3.2]{BDMTV_split}.

If~$v$ is a finite place of~$F$ not dividing~$p$, then the local
height~$h_v^\chi$ is more or less determined by the reduction type of~$X$
at~$v$ (see~\cite{bettsdogra}). Let~$L/F_v$ be a finite extension, and
let~$\cX/O_L$ be a model
of~$X_L$ which is regular and (we may assume) semistable.
Any $x\in X(L)$ extends to an $O_L$-point of~$\cX$ by the valuative criterion for properness, and hence has a reduction $\bar x\in\cX(l)$ where~$l$ is the residue field of~$L$. The fact that~$\cX$ is regular ensures that~$\bar x$ is a smooth point of the special fibre, so lies on a unique irreducible component. Regarding the relation between reduction and local heights, we will need only the following.

\begin{lemma}\label{Lemma: BettsDogra height}
    Suppose that~$x$ and~$y$ are $K_v$-points of~$X$ reducing onto the same component of the special fibre of~$\cX$. Then $h_v^\chi(x)=h_v^\chi(y)$. In particular, if~$x$ reduces onto the same component as the base point~$b$, then~$h_v^\chi(x)=0$.
\end{lemma}
\begin{proof}
    If we write $h_L^\chi$ for the local height on the base-changed
    curve~$X_L$ associated to the local character $\chi|_{L^\times}$, then we have $h_L^\chi=[L:F_v]\cdot h_v^\chi$ (from the normalisation of the isomorphism $H^2(G_v,\Q_p(1))\cong\Q_p$, see \cite[Corollary~7.2.1]{neukirch-schmidt-wingberg:cohomology_number_fields}). So it suffices to prove the result under the assumption that~$L=F_v$.

    In this case, the local height $h_v^\chi(x)$ is determined by the image
    of~$x$ under the so-called higher Albanese map $j_v\colon X(F_v) \to
    H^1(G_v,U_Z)$ from Chabauty--Kim theory, where $U_Z$ is the depth $2$ Galois-stable quotient of the $\Q_p$-pro-unipotent \'etale fundamental group of~$X_{\bar F_v}$ constructed in~\cite{balakrishnandograQC1}; see
    \cite[Lemma~12.1.1]{bettsdogra}.  

    If~$x$ and~$y$ reduce onto the same component of the special fibre, then $j_v(x)=j_v(y)$ by \cite[Theorem~1.1.2]{bettsdogra} and so we are done.
\end{proof}

\subsubsection{Local heights above $p$ }\label{Sec: htsAbovep}

The techniques and algorithms for local heights at primes $\fp\mid p$
remain largely unchanged from the case over $\Q$. Recall that we have restricted our attention to split primes ($F_{\fp}=\Q_p$) to reduce the computations of heights
above $p$ to repeating the algorithms in \cite{BDMTV_split,
balakrishnan_dogra_müller_tuitman_vonk_2023} $d$ times. We summarise the
method here to keep this exposition reasonably self-contained.
See~\cite[Section~3]{BDMTV_split} or~\cite[\S3.3.2]{BM_AWS} for more details.

For a prime $\fp\mid p$ and a mixed extension $E$ of $G_\fp$-representations as
discussed in~\S\ref{subsec:Nekovar},
Nekov\'a\v{r} constructs the local height $h_v(E)$ using the
$F_\fp$-algebra 
$$D_{\cris}(E)\coloneqq (E\otimes B_{\cris})^{G_{\fp}}\,,$$
where $B_{\cris}$ denotes the crystalline period ring of Fontaine over
$F_{\fp}=\Q_p$ \cite{fontaine1994corps}.
Then $D_{\cris}(E)$ is a mixed extension of filtered $\phi$-modules: it
comes equipped with an automorphism $\phi$ (a Frobenius structure) and a decreasing exhaustive
filtration $\Fil$, and has graded pieces $\Q_p$, 
$V_{\dR}\coloneqq H^1_{\dR}(X_{F_{\fp}})^*=D_{\cris}(V)$, and $D_{\cris}(\Q_p(1))$, analogous to the weight
filtration of mixed extensions of Galois representations. We obtain
projections $\pi_1(E)$ and $\pi_2(E)$ as in~\S\ref{Section: QC summary}.
For instance, we have 
$$
D_{\cris}(\Q_p\oplus V\oplus \Q_p(1)) = \Q_p\oplus
V_{\dR}\oplus D_{\cris}(\Q_p(1)) \eqqcolon M_0\,,
$$
where the filtration on $M_0$ is given by ${\Fil}^1 = {0}$ and 
  ${\Fil}^0 = \Q_p \oplus {\Fil} ^0 V_{\dR}$, with ${\Fil}
  ^0 V_{\dR}$ induced by the Hodge filtration by duality.
  The Frobenius structure on $M_0$ is induced by the Frobenius structure on its
  graded pieces, where the action on $\Q_p$ is trivial, the Frobenius
  structure on $V_{\dR}$ is induced from the usual Hodge filtration on
  $H^1_{\dR}(X_{F_{\fp}})$ by duality, and the action on
  $D_{\cris}(\Q_p(1))$ is multiplication by
  $1/p$.
There is a unique 
$\Q_p$-linear isomorphism
\[\lambda^{\phi}\colon  M_0\xrightarrow{\sim}
D_{\cris}(E)\,,\]
which respects the $\phi$-structure and acts as the identity on the graded pieces of the mixed extensions. Moreover, there is a $\Q_p$-linear isomorphism
\[\lambda^{\Fil}\colon M_0\xrightarrow{\sim}D_{\cris}(E)\,,\]
 which respects the Hodge filtration and also acts as the identity on the graded pieces of the mixed extensions. 
More precisely, there is a choice for $\lambda^{\Fil}$ and suitable
bases such that, in $1\times 2g\times 1$-block matrix notation,  we have
  \begin{equation}\label{E:lambdas}
  \lambda^{\phi}= \begin{pmatrix}
  1 & 0 & 0\\
  \vec{\alpha}_{\phi} & I_{2g} & 0\\
  \gamma_{\phi} & \vec{\beta}_{\phi}^{\intercal} & 1
\end{pmatrix}\,,
\quad \lambda^{\Fil}=\begin{pmatrix}
  1 & 0 & 0\\
  0 & I_{2g} & 0\\
  \gamma_{\Fil} & \vec{\beta}_{\Fil}^{\intercal} & 1
\end{pmatrix}\,.
  \end{equation}

Recall that the local height $h_\fp(E)$ depends on a homomorphism
$\chi_{\fp}\colon
F_{\fp}^{\times}\to \Q_p$
 and a splitting $s_{\fp}$ of the Hodge filtration
$H^0(X_{F_{\fp}},\Omega^1)\subseteq H^1_{\dR}(X_{F_{\fp}})$. By duality, the latter induces
projections%
\begin{equation*}
  s_1,\,s_2\colon V_{\dR}\to s_{\fp}(V_{\dR}/\Fil^0)\oplus
  \Fil^0\simeq V_{\dR}\,.
\end{equation*}
By Kummer theory, $\chi_{\fp}$ induces a homomorphism 
\[\HH_f^1(G_{\fp},
\Q_p(1)) \simeq O_{F,\fp}^*\otimes_{\hat{\Z}}\Q_p = \Z_p^*\otimes_{\hat{\Z}}\Q_p \to \Q_p\,,\]
which we also denote by $\chi_{\fp}$.
With this auxiliary data, the local height $h_{\fp}(E)$ and the
projections $\pi_i(E)$ can be written in
terms of the entries of the matrices~\eqref{E:lambdas} as follows.
See~\cite[Lemma~5.5]{BDMTV_split} and~\cite[Equation~(41)]{BBBLMSTV_padic}.
\begin{lemma}\label{Lemma :htformula}
  We have
  \[h_{\fp}(E) 
  = \chi_{\fp}(\gamma _\phi  - \gamma _{\Fil} - \vec{\beta} _\phi ^{\intercal}\cdot
  s_1(\vec{\alpha} _\phi ) - \vec{\beta} ^{\intercal}_{\Fil}\cdot s_2(\vec{\alpha} _\phi
  ))\text{.}\] 
  Moreover, with respect to the basis of $T_{\fp}$ induced by the choice of basis
  above, we have
\begin{align}\label{eq:mixedcpts}
  \pi_1(E)\otimes \pi_2(E)= {\vec{\alpha}}_{\phi}^{\intercal}\cdot \left(\begin{array}{c} I_g
 \\
  \hline 0_g \end{array} \right) \otimes  \left({\vec{\beta}}_{\phi}^{\intercal} -
  {\vec{\beta}}_{\Fil}^{\intercal}\right)\cdot \left(\begin{array}{c} 0_g \\ \hline I_g \end{array} \right). 
\end{align}

\end{lemma}

For $x\in X(F_{\fp})$ we denote
\[M(x)\coloneqq D_{\cris}(A_{\fp}(x))\,.\]
The local height $h_\fp(x)$ and the element $\pi(A_{\fp}(x))
\in T_\fp\otimes T_{\fp}$ are then determined by the structure of $M(x)$
as a filtered $\phi$-module via Lemma~\ref{Lemma :htformula}.

To compute the filtration and Frobenius structure of $M(x)$, we follow the
strategy in \cite[\S 4]{BDMTV_split} and \cite[\S 5]{BDMTV_split}
respectively. The authors of \cite{BDMTV_split} consider a unipotent vector
bundle $\mcM\coloneqq \mcA_{Z,b}$ on $X$ with connection, which has a
Frobenius structure (from its analytification, which is a unipotent
$F$-overconvergent isocrystal, via a comparison theorem of
Chiarellotto and Le Stum \cite[Prop. 2.4.1]{chiarellottoLeStum}) and a Hodge filtration, which can be
computed using universal properties proved by Hadian
\cite{hadian2011motivic} and Kim \cite{kim2009unipotent}. One then uses
Olsson's comparison theorem \cite{olssoncomparsion} in non-abelian
$p$-adic Hodge theory to show (see~\cite[Lemma~5.4]{BDMTV_split})
\[M(x)\cong x^*\mcM^{\*}.\]

We outline the steps of the computation in the following sections,
following \cite[\S 4]{BDMTV_split} and \cite[\S 3.2, \S 4.1]{balakrishnan_dogra_müller_tuitman_vonk_2023} to compute the Hodge filtration and \cite[\S4]{BDMTV_split} and \cite[\S 3.2 \S 4.2] {balakrishnan_dogra_müller_tuitman_vonk_2023} to compute the Frobenius structure. 

\subsubsection{Hodge filtration}\label{subsec:Hodge}
To compute the Hodge filtration associated to the unipotent bundle $\mcM$,
we first compute its restriction to an affine $Y\subset X$, as
unipotent bundles on $Y$ are trivial. 
Since we will use Tuitman's algorithm
\cite{tuitman2016counting, tuitman2017counting} for reduction in
rigid cohomology, we choose $Y$ to be a plane curve $Y\colon Q(x,y)=0$ that
satisfies Tuitman's conditions
(see~\cite[Assumption~3.10]{balakrishnan_dogra_müller_tuitman_vonk_2023})
for all embeddings $\sigma_{\fp}$.
In particular, $Q$ is monic in $y$.
We also assume that the base point $b$ is in $Y(F)$, and is $\fp$-integral for all
$\fp\mid p$.
Although the entries of the matrix representing
$\lambda^{\fil}$ in~\eqref{E:lambdas}
are in $F_{\fp}\simeq \Q_p$, we compute the Hodge filtration globally.

We introduce some notation to describe the action of the connection and
filtration on $\mcM|_{Y}$. Let $D=X\setminus Y$ and $\delta=\#D(\overline{F})$.  We suppose we are given meromorphic $1$-forms
  $\omega_0,\ldots,\omega_{2g-2+\delta}$ on $X$ satisfying the following four conditions:
\begin{itemize}
  \item All $\omega_i$ are $O_{\fp}$-integral for all~$\fp\mid p$. \hfill \puteqnum \label{eq:211}
  \item $\omega_i|_Y\in H^0(Y, \Omega^1)$  for all $i$.  \hfill \puteqnum \label{eq:212}
    \item The set $ \{\omega_i\}_{i=0}^{g-1}$ is a basis of
      $H^0(X,\Omega^1)$ and $\{\omega_i\}_{i=0}^{2g-1}$ is a set of
      differentials with residue zero such that their classes form a
      basis of $H^1_{\dR}(X)$ which is symplectic  with respect to the cup product
      pairing. 
      Write $\vecc{\omega}\coloneqq(\omega_i)_{i=0}^{2g-1}$. \hfill \puteqnum \label{eq:213}
    \item  The forms  $\omega_i$, where $i\in \{2g ,\ldots, 2g-2+\delta\}$,
      have at most simple poles on $X$, 
      and extend the basis of $H^1_{\dR}(X)$ to 
      a basis of $H^1_{\dR}(Y)$. \hfill \puteqnum \label{eq:214}
\end{itemize}    

Using the work of Tuitman \cite{tuitman2016counting, tuitman2017counting}, we can compute differentials $\omega_i$ that satisfy these conditions, except that the first $2g$ might not be symplectic. This can be easily rectified using linear algebra.
We also suppose that we have a matrix $Z\in M_{2g \times 2g}(F)$ that represents the action of our
  chosen cycle $Z\in
  \ker(\NS(J)\to \NS(X))$ on $H^1_{\dR}(X)$ with respect to 
  $\vecc{\omega}$.

  \begin{Remark}\label{R:Z}
  When $X$ is a modular curve, then, as discussed
  in~\cite[\S3.5.2]{balakrishnan_dogra_müller_tuitman_vonk_2023}, we may
  use a suitable linear combination of powers of the Hecke operator $T_p$, which we may compute using the Eichler--Shimura relation, see \S \ref{Subsection: Correspondences}. If~$X$ is not a modular curve, then one can instead compute $\End(J)$ using an algorithm of
    Costa--Mascot--Sijsling--Voight~\cite{CMSV_endo}. The output of this algorithm includes a description of the action of the endomorphism ring on~$H^0(X,\Omega^1)$, but not on all of~$H^1_{\dR}(X)$ as we would need. Computing the action on all of~$H^1_{\dR}(X)$ can sometimes be tricky in practice, see \cite{bettslocalheights}.
  \end{Remark}

\begin{algo}[Computing the Hodge filtration]\label{Algo: Hodge filtration}
 \hfill 

  \textit{Input}: Differentials $\omega_0,\ldots, \omega_{2g-2+\delta}$ as
  in \eqref{eq:211}-\eqref{eq:214}
and a matrix $Z\in M_{2g\times2g}(F)$  representing  the action of a nontrivial cycle in
  $\ker(\NS(J)\to \NS(X))$ on $H^1_{\dR}(X)$ with respect to $\vec{\omega}$.
   \\ \textit{Output}: $\gamma_{\Fil,\mcO(Y)}\in F(1)\otimes \mcO(Y) $ and
   $\vec{\beta}_{\Fil,\mcO(Y)}=(0_g, b_{\Fil,\mcO(Y)})\in F^{2g}$,
   representing the restriction of $\lambda^{\Fil}$ in~\eqref{E:lambdas} to $Y$.

\begin{enumerate}
\item For each $x \in D(\overline{F})$, do the following: \label{Algo Hodge: Step 1}
  \begin{enumerate}
    \item Compute a uniformiser $t_x$. \label{Algo Hodge: Step 1a}
\item Expand all $\omega_i$ into
  power series $\omega_{x,i} \in F(x)[[t_x]]$ and set $\vecc{\omega}_x \coloneqq
      (\omega_{x,i})^{2g-1}_{i=0}$.
 \label{Algo Hodge: Step 1b}
    \item Integrate formally to compute 
      $\overrightarrow{\Omega_x}$ such that $d\overrightarrow{\Omega_x}=\overrightarrow{\omega_x}$.\label{Algo Hodge: Step 1c}
  \end{enumerate}
\item Solve   
 for $\eta=\sum_{i=2g}^{2g-2+\delta}e_i\omega_i$    such that for all $x\in D(\overline{F})$ we have 
    \[
   \Res\left(d\vecc{\Omega_x}^{\intercal}
   Z\vecc{\Omega}_x-\sum^{2g-2+\delta}_{i=2g}e_i\omega_{x,i}\right)=0\,.
 \]
 \label{Algo Hodge: Step 2}
 \item For each $x \in D(\overline{F})$ integrate formally to
   compute $g_x\in
   F(x)((t_x))$ such
   that \[dg_x\ = d\vecc{\Omega}_x^{\intercal} Z\vecc{\Omega}_x-\sum_i
    e_i\omega_{x,i}\,.\]
 \label{Algo Hodge: Step 3}
 \item Let $N=(0_g \hspace{0.1 cm} I_g)^{\intercal}\in M_{2g\times g}(\mcO(Y))$.
   Solve for $b_{\Fil}=(b_g,\ldots, b_{2g-1}) \in F^{g}$ and $\gamma_{\Fil} \in \mcO(Y)$ such that  $\gamma_{\Fil}(b)=0$
    and such that
 
   \[
     g_x + \gamma_{\Fil} - b_{\Fil}^{\intercal}N^{\intercal}\vecc{\Omega}_x -
     \vecc{\Omega}_xZNN^{\intercal} \vecc{\Omega}_x 
   \]

 has trivial principal part for all $x$. \label{Algo Hodge: Step 4}
 \item Return $\gamma_{\Fil,\mcO(Y)}\in F(1)\otimes \mcO(Y) $ and
   $\vec{\beta}_{\Fil,\mcO(Y)}=(0_g, b_{\Fil,\mcO(Y)})\in F^{2g}$. \label{Algo Hodge: Step 5}
\end{enumerate}
\end{algo}

\begin{prop}
  Algorithm~\ref{Algo: Hodge filtration} computes the unique
  filtration-preserving splitting of $M(x)$ of the
  form~\eqref{E:lambdas}. More precisely, for each $\fp | p$, the matrix 
  \[\begin{pmatrix}
  1 & 0 & 0\\
  0 & I_{2g} & 0\\
    \sigma_{\fp}(\gamma_{\Fil,\mcO(Y)}) & \sigma_{\fp}(\vec{\beta}_{\Fil,\mcO(Y)})^{\intercal} & 1
  \end{pmatrix}\]
  represents the restriction to $Y$ of the unique filtration-preserving
  splitting of $M(x)$ of the
  form~\eqref{E:lambdas}.

\end{prop}

\begin{proof}
  Since Algorithm~\ref{Algo: Hodge filtration} is essentially the same as
\cite[Algorithm 5.20]{BM_AWS}, the result follows from~\cite[Lemma~4.10, Theorem 4.11]{BDMTV_split}. 
Namely, the existence (and uniqueness) of the $e_i$ in Step~\eqref{Algo Hodge: Step
  2} follow from~\cite[Lemma 4.10]{BDMTV_split}. By~\cite[Proposition
  4.11]{BDMTV_split}, the functions $\gamma_{\Fil}(x)$ and the vector
  $\vec{\beta}_{\Fil}$ give 
  $\lambda^{\Fil}$ as in~\eqref{E:lambdas}.
\end{proof}

\begin{Remark}\label{R:HodgeTrick}
  The analogue of Algorithm~\ref{Algo: Hodge filtration} 
in~\cite{BDMTV_split}
and~\cite{balakrishnan_dogra_müller_tuitman_vonk_2023}
was formulated
entirely over a field $L$ such that 
$D(L) = D(\overline{F})$. Working over such a field
was sufficient for all quadratic Chabauty
examples computed so far. However, even computing the
uniformisers $t_x$ over $L$ becomes difficult when the degree of $L$ is large, as is
the case for the curve $X'_H$.
Steps \eqref{Algo Hodge: Step 1} and \eqref{Algo Hodge: Step 3} only require working with one $x$ (or rather one
$t_x$) at a time, so we can work over the respective residue fields $F(x)$.
For Steps \eqref{Algo Hodge: Step 2} and \eqref{Algo Hodge: Step 4} we combine the results for the individual points $x$ and
apply simple linear algebra; here we embed the fields $F(x)$ into a field $L$ as
above.
\end{Remark}

\subsubsection{Frobenius structure}\label{subsec:Frob}
To compute the Frobenius action on $\mcM$ and $M(x)$, \cite{BDMTV_split}
uses the
theory of unipotent isocrystals. We refer the reader to \cite[Appendix
A.2]{BDMTV_split} for an overview of the underlying theory
and~\cite[Section~5]{BDMTV_split},~\cite[\S4.2]{balakrishnan_dogra_müller_tuitman_vonk_2023},
and~\cite[\S5.2.3]{BM_AWS} for details of the construction.

The crucial ingredient is Tuitman's reduction algorithm 
\cite{tuitman2016counting, tuitman2017counting}. Given a
meromorphic
differential $\xi$ on $X_{F_{\fp}}$ without residues, the algorithm produces a vector of constants
$\vec{c}\in F_{\fp}^{2g}$ and an overconvergent function 
$H$ such that
\begin{equation}\label{E:redn}
  \vec{c}\,{}^{\intercal} \vecc{\omega} + dH  = \xi\,.
\end{equation}

We will 
  normalise the function $H$ by requiring
  $H(b_0)=0$, where
$b_0$ is the unique Frobenius-invariant point in the residue disc of $b$.

In particular, we may use this to compute the action of Frobenius $\phi$ on
$H^1_{\rig}(X _{ F_{\fp}})\cong H^1_{\dR}(X_{F_{\fp}})$:
  \begin{equation}\label{E:Frobomega}
    \phi^*\vecc{\omega} =  \Frob \vecc{\omega} + d\vec{f}\,,\quad \Frob \in
    M_{2g\times2g}(F_{\fp})\,,\; f_0(b_0)=\cdots=f_{2g-1}(b_0)=0\,.
  \end{equation}

  This gives an overconvergent lift of Frobenius on a wide open subspace
  $\mathcal{U}$ of (the analytification of)
  $X_{F_{\fp}}$, containing all residue discs reducing to a point in $Y(\F_p)$
  such that the map $(x,y)\to x$ to $\P^1$ is unramified at this point. Following~\cite{balakrishnantuitman}, we call such residue discs
  \textit{good} (with respect to $Y$), and we call a residue polydisc $\resdisc{x_1} \times \cdots\times
  \resdisc{x_d} \subset X(F\otimes \Q_p)$ \textit{good}  (with respect to $Y$) if all $\resdisc{x_i}$ are good.

In order to find the Frobenius-equivariant splitting of $M(x)$ for $x$ in a good
residue disc $\resdisc{x}$ (see \cite[\S 5.3.2]{BDMTV_split} and \cite[\S 4.2]{balakrishnan_dogra_müller_tuitman_vonk_2023}), one first
computes the Frobenius-equivariant splitting of $M(x_0)$ for the unique Frobenius-invariant point
$x_0\in \resdisc{x}$, and then uses parallel transport to compute the Frobenius-equivariant splitting for $M(x)$. 
To do this, one first applies Tuitman's algorithm to the differential (which has no residues on $X$):
\begin{equation}\label{E:xiexpansion}
  \xi \colonequals(\phi^*\vec{\omega}^{\intercal})Z\vec{f} + (\phi^*\eta
  - p \eta) +(\Frob^{\intercal}Z  \vec{f})^{\intercal} \vec{\omega}\,,
\end{equation}
where $\eta$ is as in Algorithm~\ref{Algo: Hodge filtration}, 
resulting in $\vec{c}$ and $H$ as in~\eqref{E:redn}, 
and then sets 
\begin{equation}\label{E:g}
\vec{g}\colonequals -\Frob^{\intercal} Z  \vec{f} +\vec{c}\,.
\end{equation}

One can then express the Frobenius-equivariant splitting $M(x_0)\to
\Q_p\oplus V_{\dR}\oplus D_{\cris}(\Q_p(1))$ with respect to the basis
$\vecc{\omega}$ as shown in \cite[\S5.3.2]{BDMTV_split}:
\begin{center}
\begin{equation}
        \lambda^{\phi}(x_0) = \begin{pmatrix}
 1 & 0 & 0 \\
          (I_{2g} - \Frob)^{-1} \vec{f} & I_{2g} & 0\\
          \frac{1}{1-p}(\vec{g}^{\intercal}(I_{2g} - \Frob)^{-1}\vec{f} + H) & \vec{g}^{\intercal}(\Frob - p)^{-1} & 1
    \end{pmatrix}(x_0).
\end{equation}
\end{center}

For two points $x_1,x_2$ in the same residue disc, one sets
\begin{center}
    \begin{equation}\label{E:PT} I^{\pm}(x_1,x_2)\coloneqq \begin{pmatrix}
    1 & 0 & 0\\
      \int^{x_2}_{x_1} \vecc{\omega} & I_{2g} & 0\\
    \int^{x_2}_{x_1}\eta\pm\int^{x_1}_{x_2}\vecc{\omega}^{\intercal}Z\vecc{\omega} & \pm \int^{x_2}_{x_1}\vecc{\omega}^{\intercal}Z & 1
\end{pmatrix}.
\end{equation}
\end{center}
Then $I^+(x_1,x_2)$ (respectively $I^-(x_2,x_1)$) represents parallel
transport of the connection from $x_1$ to $x_2$ (respectively from $x_2$ to
$x_1$).

This results in the following algorithm to compute $\lambda^{\phi}$
(see~\cite[Algorithm~5.26]{BM_AWS}):
\begin{algo}[Computing the Frobenius structure]\label{Algorithm: Frobenius structure}\hfill

  \textit{Input}: Differentials $\omega_0,\ldots, \omega_{2g-2+\delta}$ as
  in \eqref{eq:211}-\eqref{eq:214}
and a matrix $Z\in M_{2g\times 2g}(F)$  representing  the action of a nontrivial cycle in
  $\ker(\NS(J)\to \NS(X))$ on $H^1_{\dR}(X)$ with respect to $\vec{\omega}$. The differential $\eta$ as in
  Algorithm~\ref{Algo: Hodge filtration} and a point $x\in Y(F_\fp)$ in a
  good disc.
  \\\textit{Output}: The constant $\gamma^{\phi}$, the row vector
  $\vec{\beta}^{\phi}$ with $2g$ entries, and the column vector
  $\vec{\alpha}^{\phi}$ with $2g$ entries as in~\eqref{E:lambdas}.
\begin{enumerate}
  \item Use Tuitman's algorithm to compute $\Frob$ and $\vec{f}$ as
    in~\eqref{E:Frobomega}.

  \item Apply Tuitman's algorithm to $\xi$ as in~\eqref{E:xiexpansion} to
    compute $\vec{c}$ and $H$ as in~\eqref{E:redn}.
\item Define $\vec{g}$ as in \eqref{E:g} and set \[M(b_0,x_0)\coloneqq \begin{pmatrix}
  1 & 0 & 0\\
  (I-\Frob)^{-1}\vec{f} & I_{2g} & 0\\
  \frac{1}{1-p}\vec{g}^{\intercal}(I-\Frob)^{-1}\vec{f}+\vec{h} & \vec{g}^{\intercal}(\Frob-p)^{-1} & 1
\end{pmatrix}(x_0).\]

\item 
Compute the matrix
\[I^+(x, x_0)I^-(b_0, b)M(b_0,x_0)\eqqcolon\begin{pmatrix}
    1  & 0 & 0\\
    \vec{\alpha}_{\phi} & I_{2g} & 0\\
    \gamma_{\phi} & \vec{\beta}_{\phi} & 1 
 \end{pmatrix}.\]
    and return $\vec{\alpha}_{\phi}, \vec{\beta}_{\phi}, \gamma_{\phi}$.
\end{enumerate}

\end{algo}
The discussion in~\cite[Section~5]{BDMTV_split} shows the following:
\begin{prop}
  Let $\gamma_{\phi}$, $\vec{\beta}_{\phi}$, and $\vec{\alpha}_{\phi}$ be
  as returned by  Algorithm~\ref{Algorithm: Frobenius structure}. Then
  these give the unique Frobenius-invariant splitting of $M(x)$ 
  as in~\eqref{E:lambdas}.

\end{prop}
We can also use Algorithm~\ref{Algorithm: Frobenius structure} to compute $M(x)$
for a parametric point in a good disc. In the latter case, the algorithm
returns power series in a local parameter of the disc.

\subsection{Solving for the height pairing}\label{subsec:solving}
We now discuss how to determine the height pairing $h\colon (J(F)\otimes\Q_p)\times(J(F)\otimes\Q_p)\to\Q_p$, which we are reinterpreting as a $\Q_p$-linear map $T\otimes_{\Q_p}T\to\Q_p$ using Assumption~\ref{A:log}. For a point $x\in X(F\otimes\Q_p)$, we write
\[
e(x) \coloneqq \log([x-b])\otimes\log(E_Z([x-b])+c_Z) \in T\otimes_{\Q_p}T \,.
\]
Then we have 
$e(x) = \pi_1(A_{Z,b}(x))\otimes \pi_2(A_{Z,b}(x))$ by construction.
Suppose that we have found some number~$x_1,\dots,x_N$ of $F$-rational points on~$X$. Using Lemma~\ref{Lemma :htformula}, we can compute the points~$e(x_i)$, as well as the local heights $h_{Z,\fp}(x_i)$ for all~$\fp\mid p$. So if we can compute the remaining local heights $h_{Z,v}(x_i)$ for~$v\nmid p$ (e.g.~if they are all zero, as will be the case in our example), then we can compute the global height $h_Z(x_i)$. Since~$h_Z(x_i)=h(e(x_i))$, each computed value of~$e(x_i)$ and~$h_Z(x_i)$ imposes one linear constraint on the height $h\colon T\otimes_{\Q_p}T\to\Q_p$, and so we can uniquely determine the height pairing~$h$ as soon as the elements~$e(x_i)$ span $T\otimes_{\Q_p}T$.

The disadvantage of this approach is that we may need a large number of known
rational points. 
In order to mitigate this difficulty, we use
the following consequence of the discussion in \cite[Section
3.4]{balakrishnandograQC2}:
\begin{lemma}\label{L:equivsplit}
  If the splitting $s_{\fp}$ of the Hodge filtration on
  $H^1_{\dR}(X_{F_{\fp}})$ commutes with
  $\End(J)$ for all $\fp \mid p$,
then we have $h(\alpha x,y) = h(x,\alpha y)$ for all $\alpha \in \End(J)$ and $x,y \in J(F)$.
\end{lemma}
Given any splitting
$s_{\fp}$ of the Hodge filtration at $\fp$ and generators of $\End(J)$
acting on $H^1_{\dR}(X_{F_{\fp}})$, it is a matter of simple linear algebra
to construct a splitting as in the lemma. In particular, such splittings
exist.

Lemma~\ref{L:equivsplit} implies that $h$ factors through a linear
functional on $\mcE\coloneqq T\otimes_{K\otimes\Q_p}T$,
where~$K=\End^0(J)$. The space~$\mcE$ is typically of smaller dimension
than that of~$T\otimes_{\Q_p}T$, and so we require fewer $F$-rational points~$x_i$ to determine~$h$ in this case.
 
For example, if $K$ is a division algebra of dimension $D$ (as is true for the curve we
are interested in), then its left-action on the $dg$-dimensional
$\Q$-vector space $T_0\coloneqq H^0(X,\Omega^1)^*$ is free, and thus $T_0$
is a free left $K$-module of rank $dg/D$. Hence $\mcE_0\coloneqq
T_0\otimes_K T_0$ is a free $K$-module of dimension $(dg/D)^2$, and
$\mcE\coloneqq\mcE_0\otimes \Q_p$ is a free $K\otimes \Q_p$-module of rank
$(dg/D)^2$. Hence, the space $\mcE^*$ of functionals we are interested in
has dimension $(dg/D)^2$ over $K\otimes \Q_p$ and dimension $(dg)^2/D$ over
$\Q_p$, in contrast to $\dim_{\Q_p}T\otimes_{\Q_p}T = (dg)^2$.
\begin{Remark}\label{R:moddim}

In the case of modular curves, including our example in Section~\ref{sec3}, we often have $D=g$. Thus the space $\mcE^*$ of functionals we are interested in has dimension $d^2g$.
\end{Remark}

\begin{Remark}\label{Remark : More correspondences}
One may reduce the number of required $F$-rational points further
  by noting that the mixed extensions $A(x)$ and 
$A_\fp(x)$ depend on the choice of cycle $Z$. Therefore the same point $x\in X(F)$ 
  gives rise to $r_{\NS}-1$ elements $e(x) \in \mcE$. 
\end{Remark}

\begin{Remark}\label{R:CG}
  An alternative basis of $\mcE^*$ is given by suitable products of
  abelian logarithms~\eqref{log2}. In particular, we may attempt to use this when we
  do not have enough rational points $x_i$ such that the $e(x_i)$ form a
  basis of $\mcE$. However, to solve for the height pairing, we then have to evaluate the global height pairing between 
  points on the Jacobian that are not of the form $([x-b],E_Z([x-b])+c_Z)$ for $x\in X(F)$.
  At present, no algorithm is known to compute such heights using 
  Nekov\'a\v{r}'s construction. Besser has shown~\cite{besserCGNek} that
  Nekov\'a\v{r}'s construction is equivalent to a construction due to
  Coleman and Gross~\cite{colemangross}. For hyperelliptic curves, there are
  efficient algorithms to compute Coleman--Gross heights between divisors,
  see~\cite{balakrishnanbesserhyperelliptic}
  and~\cite{GM24}. This can then be used to solve for the
  $\alpha_j$.
  See~\cite[\S3.3]{balakrishnan_dogra_müller_tuitman_vonk_2023} for
  details. 
\end{Remark}

\begin{Remark}
  We cannot get additional relations by varying the splittings $s_{\fp}$,
  see~\cite[Remark~4.9]{balakrishnandograQC1}.
\end{Remark}

\subsection{Root finding}\label{subsec:Hensel} Given a system of multivariate $p$-adic power series in $d$ variables, our main tool to find solutions is to apply a multivariate Hensel's lemma, as was done in \cite[Theorem A.1]{BBBMnumberfields} (see also \cite[Theorem 4.1]{conrad:multihensel}). 

We follow the strategy outlined in \cite[Algorithm 1]{BBBMnumberfields}: first we compute an appropriate truncation of each multivariate $p$-adic power series, to reduce to the problem of root-finding of multivariate $p$-adic polynomials. Then we naively compute all solutions to the system mod $p$. For each mod $p$ solution, we attempt to apply multivariate Hensel to lift it to a solution in $\Z/p^N\Z$, where $N$ denotes the desired $p$-adic precision.  If the hypothesis of the multivariate Hensel's lemma is not satisfied, i.e., that roots are present  with multiplicity, we attempt to lift naively to sufficiently many digits (producing a solution in $\Z/p^{N'}\Z$ where $N' \geq 3$)  to separate each root with multiplicity, and, if successful and $N' < N$, re-try an application of multivariate Hensel.  

Note that in some cases, multiple roots can persist through precision $\Z/p^N\Z$. For instance, this may happen in the presence of nontrivial automorphisms of the curve. In this case, one should use the automorphisms to strategically choose local coordinates and re-parametrise the system, as in \cite[Lemma A.3]{BBBMnumberfields}.

\subsection{Precision analysis}\label{subsec:precision}

We now discuss the loss of precision in our method. For simplicity, we
assume that $\Upsilon=\{0\}$. We also assume that we have found points
$x_1,\ldots,x_N\in Y(F)$ such that all $\sigma(x_i)$ lie in good residue polydiscs and such that
$e(x_1),\ldots,e(x_N)$ form a basis of $\mcE=T\otimes_{K\otimes\Q_p}T$. Hence we get a dual basis 
$\Psi_1,\ldots,\Psi_N$ of $\mcE^*$, so we may write 
\begin{equation}\label{E:dj}
h=\sum_j d_j\Psi_j\,;\quad d_j\in \Q_p\,.
\end{equation}
Since $p$ splits in $O_F$, all local computations take place over $\Q_p$, and hence the analysis is almost the same
as in~\cite[Section~4]{balakrishnan_dogra_müller_tuitman_vonk_2023}, which
deals with the case $F=\Q$. However, we found some minor inaccuracies in loc. cit., which we fix below.

\subsubsection{Precision loss in height computations}\label{subsec:prechts}
The analysis of the possible loss of $p$-adic precision when computing
$h_{\fp}(x)$ for a fixed point $x \in Y(F_{\fp})$ in a good disc is exactly
the same as
in~\cite[\S4.1, \S4.2]{balakrishnan_dogra_müller_tuitman_vonk_2023}.
Consequently, this allows us to bound the precision loss in computing
$h(x)$ for $x\in Y(F)$ such that $\sigma(x)$ is in a good residue polydisc for
all $\fp\mid p$. Since $e(x)$ can be obtained using the same data as
$h_{\fp}(\sigma_\fp(x))$ by Lemma~\ref{Lemma :htformula}, the same analysis
allows us to bound the loss of precision in computing the coefficients
$d_j$ in~\eqref{E:dj}.

\subsubsection{Bounding valuations of power series coefficients}\label{subsec:boundvals}
In order to solve for its zeroes, we need to expand the function $\rho$
into
multivariate power series in good residue polydiscs $\mathcal{D}\coloneqq
\resdisc{x_1} \times \cdots \times \resdisc{x_d}$:
\begin{equation}\label{}
  \rho = \rho(\vec{t}) = h(\vec{t}) - \sum_{\fp\mid p} h_{\fp}(\vec{t})\in
  \Q_p[[\vec{t}]]\,,
\end{equation}
where $\vec{t} = (t_1,\ldots,t_d)$ and $t_j$ is a local parameter in 
$\resdisc{x_j}$.
Given a $\vec{t}$-adic truncation level $M>0$ and a starting precision $N_1>0$ for our
$p$-adic computations,~\S\ref{subsec:prechts} allows us to provably approximate
the coefficients of these power series
to some $p$-adic precision $N_2<N_1$. As we
will discuss in~\S\ref{subsec:prechensel}, in order to approximate the common
zeroes of $d$ quadratic Chabauty functions $\rho$, we need a
lower bound on the valuations of the coefficients. We now explain how this
can be obtained. Again, we 
follow~\cite[\S4.4]{balakrishnan_dogra_müller_tuitman_vonk_2023}, but in
doing so we will also correct a minor inaccuracy in loc. cit.

For a univariate power series $S$ over $\Q_p$ and $i\in \Z$, we let $\tau_i(S)$ be the
valuation of the $i$th coefficient of $S$. We extend this to
$\tau_{\vec{i}}(S)$ for $S\in \Q_p[[\vec{t}]]$, and we set, in this case
$\tau_i(S) \coloneqq \min_{\sum\vec{i}=i}\tau_{\vec{i}}(S)$. Finally, we
write $\tau_i(S_1,\ldots,S_d)\coloneqq \min\{\tau_i(S_1),\ldots,\tau_i(S_d)\}$ for $S_1,\ldots,S_d\in \Q_p[[\vec{t}]]^d$.

Recall the formula for the local height $h_{\fp}(x_j)$ 
in Lemma~\ref{Lemma :htformula} and, in particular, Algorithm~\ref{Algorithm: Frobenius
structure}. As in~\cite[Section~4]{balakrishnan_dogra_müller_tuitman_vonk_2023}, define 
\[
c_1\colonequals \ord_p(\lambda^\phi(x_j'))\,,
\]
where $x'_j$ is the fixed point of Frobenius
in the disc of $x_j$ and we take $t_j$ such that $t_j(x'_j)=0$. 
 Define 
\[
c_2 \colonequals \min\{0,v_{\spl}, \ord_p(\beta_{\Fil}),v_{\spl}
+\ord_p(\beta_{\fil}) \}\,,
\]
where
$v_{\spl}$ is the minimum of the valuations of the coefficients expressing our
chosen equivariant splitting in terms of 
$\vec{\omega}$ and
$\ord_p(\gamma_{\Fil})$ is the minimum of the valuations of the
coefficients of $\gamma_{\Fil}$.

Our goal is to find lower bounds for $\tau_i(\rho)$. These simplify if we
only consider sufficiently large $i$; more precisely, we assume that $i\ge
i_0$, where $i_0$ is defined
in~\cite[\S4.4]{balakrishnan_dogra_müller_tuitman_vonk_2023}.
This is permissible, since in practice $i_0$ is usually very small.
Under this assumption, the discussion in~\cite[\S4.4]{balakrishnan_dogra_müller_tuitman_vonk_2023} shows that  
\begin{equation}\label{tauihfp}
  \tau_i(h_\fp)\ge -2\lfloor\log_pi\rfloor + c_1+c_2,
\end{equation}
where the term $-2\lfloor\log_pi\rfloor$ comes from the double integral in the parallel transport matrix~\eqref{E:PT}.
\begin{Remark}\label{R:c1}
For $F=\Q$,~\cite[Lemma~4.5]{balakrishnan_dogra_müller_tuitman_vonk_2023}
claims that $\tau_i(h_p) \ge -2\lfloor\log_pi\rfloor + c_2$, but 
  it is easy to see that the $c_1$ term in~\eqref{tauihfp} is in fact necessary due
to~\cite[(4.18)]{balakrishnan_dogra_müller_tuitman_vonk_2023}. 
\end{Remark}

For general $F$, the constant $c_2$ actually depends on the prime $\fp$, so we define $c_2$ to be the
minimum of the $c_2$ values for all $\fp\mid p$.
Similarly, we take $c_1$ to be the minimum of the $c_1$ values in the residue
discs $\resdisc{x_j}$.
Therefore we obtain the lower bound
\begin{equation}\label{E:tauihp}
  \tau_i\left(\sum_{\fp\mid p}h_\fp\right)\ge -2\lfloor\log_pi\rfloor + c_1+c_2\,.
\end{equation}

By~\eqref{E:dj}, bounding $\tau_i(h)$ reduces to bounding the valuations of the coefficients of the expansions of the entries of $e(x)$ for $x\in \mathcal{D}$. As stated in~\cite[\S4.4]{balakrishnan_dogra_müller_tuitman_vonk_2023}, this is very similar to bounding $\tau_i(h_\fp)$, since $e$ and $h_\fp$ are determined by the same matrices $\lambda^{\phi}$ as in~\eqref{E:lambdas} by Lemma~\ref{Lemma :htformula}. The main difference is that for the entries of $e$, we need to consider the product of two matrices of this kind, rather than only one, as is the case for $h_\fp$. Nevertheless, the dependency of $\tau_i(h)$ on $i$ is the same as that of $\tau_i(h_{\fp})$, because only single integrals appear in the formula for $e$ in Lemma~\ref{Lemma :htformula}. More precisely, for $i\ge i_0$, we obtain
\begin{equation}\label{E:tauih}
  \tau_i(h)\ge -2\lfloor\log_pi\rfloor + 2c_1+c_3\,,
\end{equation}
where $c_3 = \min_j\{v_p(d_j)\}$ for the $d_j$ defined in~\eqref{E:dj}.
We can now bound $\tau_i(\rho)$ as follows:
\begin{prop}\label{P:tauirho}
For $i\ge i_0$, the expansion of the function $\rho$ in a good residue polydisc satisfies
\[\tau_i(\rho) \ge -2\lfloor\log_pi\rfloor + c_1+\min\{c_2,c_1+c_3\}\,.\]
\end{prop}
\begin{proof}
This follows from~\eqref{E:tauihp} and from~\eqref{E:tauih}.
\end{proof}

\begin{Remark}
Proposition~\ref{P:tauirho} corrects~\cite[Proposition~4.6]{balakrishnan_dogra_müller_tuitman_vonk_2023} 
which claims that for $F=\Q$, we have
$\tau_i(\rho) \ge -2\lfloor\log_pi\rfloor + c_1+\min\{c_2,c_3\}$; see
  Remark~\ref{R:c1}. This does not impact the validity of any quadratic
  Chabauty examples in the literature computed using the code~\cite{QCMod}.
\end{Remark}
\subsubsection{Precision of root finding}\label{subsec:prechensel}
Suppose we have $d$ power series $\rho_1,\ldots,\rho_d\in \Q_p[[\vec{t}]]$, and
we are looking for their common roots in $(p\Z_p)^d$.

Let $N>0$ and $M>0$ be integers such that for all $i\ge M$
and for all $j\in \{1,\ldots,d\}$ we have
\begin{equation}\label{bd1}
  di+\tau_i(\rho_j)\ge N\,.
\end{equation}
We may scale so that all $\rho_j\in \Z_p[[\vec{t}]]\setminus p\Z_p[[\vec{t}]]$.
Then a root $\vec{z} \in \Z_p^d$ of the system 
\begin{equation}\label{pssystem}
  \rho_1(p\vec{t})=\rho_2(p\vec{t})=\cdots=\rho_d(p\vec{t})=0
\end{equation}
yields a root $\vec{z}\mod p^n \in (\Z/p^n\Z)^d$ of 
\begin{equation}\label{polysystem}
  \tilde{\rho}_1(p\vec{t})\equiv
  \tilde{\rho}_2(p\vec{t})\equiv\cdots\equiv \tilde{\rho}_d(p\vec{t})\equiv
  0\pmod{p^n}\,,
\end{equation}
where $\tilde{\rho}_j \coloneqq \rho_j+O(\vec{t})^M \in
\Q_p[\vec{t}]^d$.

As discussed in~\cite[Appendix~A]{BBBMnumberfields}, if the condition for applying the multivariate version of
Hensel's lemma (\cite[Theorem A.1]{BBBMnumberfields}) is satisfied, then  
a root of~\eqref{polysystem} lifts to a unique root of~\eqref{pssystem}. So
we can determine such roots of~\eqref{pssystem} to precision
$p^N$ using the approach sketched in~\S\ref{subsec:Hensel}, as long as we approximate our $p$-adic power series to precision 
$\vec{t}^M$ for $M$ satisfying~\eqref{bd1} for all $i\ge M$ and all $j$.
In
practice, we start our quadratic Chabauty computations with some $\vec{t}$-adic precision
$\vec{t}^M$ and a $p$-adic precision $p^{N_1}$. As discussed
in~\S\ref{subsec:prechts}, we obtain $N_2\le N_1$ such that 
the coefficients of $\rho_j+O(\vec{t})^M$ are computed correctly to
$N_2$ digits of $p$-adic precision. To guarantee that the roots
of~\eqref{pssystem} satisfying the Hensel condition are computed correctly
modulo $p^{N_2}$, it suffices to check that~\eqref{bd1} holds for all $j$
and all $i\ge M$, which is straightforward from Proposition~\ref{P:tauirho}. 

We do not discuss the case when the Hensel condition is not satisfied,
since this is not required for our example in Section~\ref{sec3}.

\subsection{Implementation}\label{subsec:}
The algorithms discussed above have been implemented in the computer
algebra system {\tt Magma}~\cite{MR1484478}, and the code is available
at~\cite{QCMod-NF}. It is based on the implementation of the algorithms
in~\cite{balakrishnan_dogra_müller_tuitman_vonk_2023} available at~\cite{QCMod}.
Our code incorporates the precision analysis in~\eqref{subsec:precision},
so that the result is provably correct to a certain explicit $p$-adic
precision. The root-finding is done using code that we ported from Francesca Bianchi's {\tt SageMath} implementation of~\cite[Algorithm 1]{BBBMnumberfields}, available at \cite{Bianchi_Roots}.

At present, our code is restricted to imaginary quadratic $F$.
However, an extension to real quadratic $F$ would be trivial and an
extension to other number fields would be fairly straightforward. One issue
is the choice of suitable id\`ele class characters, see also~\cite[\S2]{BBBMnumberfields}.

Our implementation assumes that the local heights $h_v$ are trivial for
$v\nmid p$, since this holds for our example $X'_H$, as shown
in~\S\ref{Sec: Hts for X away}. This means that $\Upsilon=\{0\}$.
However, it would be easy to allow for 
nontrivial $\Upsilon$, as long as we can compute it explicitly.
An algorithm that computes $\Upsilon$ for hyperelliptic curves in odd residue characteristic is discussed in~\cite{bettslocalheights}.  One can also sometimes compute $\Upsilon$ with much less work, see~\cite[Examples 5.18,
5.19]{balakrishnan_dogra_müller_tuitman_vonk_2023}.

\subsection{Assumptions, extensions, and relation to other
work}\label{subsec:extensions}

In this subsection we discuss Assumption \ref{A:log}, the scope of
quadratic Chabauty over number fields without this assumption, and how our work relates to the literature.

 Quadratic Chabauty 
 is a special case of Kim's nonabelian Chabauty program \cite{kim2005motivic, kim2009unipotent}, which generalises Chabauty's method using
 tools from $p$-adic Hodge theory and nonabelian Galois cohomology. 
 The work of Balakrishnan and Dogra~\cite{balakrishnandograQC1}
 gives an explicit realisation of Kim's approach using $p$-adic heights when
 $F=\Q$ and Assumption~\ref{A:log} is satisfied, or when $F$ is an
 imaginary quadratic field and a similar condition holds.
 While Kim and Balakrishnan--Dogra only use one prime
 $\fp\mid p$, Dogra~\cite{dogra2023unlikely} and Hast~\cite{Has21} employ a restriction of
 scalars approach that uses all primes above $p$, much like Siksek's
 extension of Chabauty's method discussed in~\S\ref{subsec:Chabauty}.
 We refer to~\cite[\S3.1]{dogra2023unlikely} for details.

 As is the case for linear Chabauty, proving finiteness of Chabauty--Kim sets over number fields can be tricky; see the discussion in~\cite{dogra2023unlikely, Has21}.
This is less relevant for explicit methods;
if we are able to compute the common zero locus, and it is finite, then this suffices for us.

We now discuss possible extensions when Assumption~\ref{A:log} is not
satisfied, as well as their attendant difficulties.
Recall that we need $d$ locally analytic functions to hope to get a finite subset as in \eqref{Eqn: QC intersections}. 
The number of independent id\`ele class characters~$\chi$ 
is at least $r_2+1$, where $(r_1,r_2)$ is the signature of $F$,
and is exactly $r_2+1$ if Leopoldt's conjecture \cite{Leopoldt1962} holds for $F$
(see~\cite[Corollary~2.4]{BBBMnumberfields}). The number of
independent trace~$0$ cycles~$Z$ is $r_{\NS}-1$. Hence we can construct
  $m\geq(r_{\NS}-1)(r_2+1)$ independent height functions $h^{(j)} =
  h^{\chi_j}_{Z_j}$ on~$X(F)$, whose local components at places~$\fp\mid p$
  are Coleman functions on~$X(F\otimes\Q_p)$, and $x\mapsto h^{(j)}(x)-\sum_{\fp\mid
  p}h^{(j)}_{\fp}(x)$ takes values in a finite set~$\Upsilon^{(j)} =
  \Upsilon_{Z_j}^{\chi^{(j)}}$ when evaluated on~$X(F)$.

The approach of quadratic Chabauty in this more general setting is to look
  for polynomial combinations of these local height functions and linear
  Chabauty functions (Coleman integrals of~$1$-forms) that vanish
  on~$X(F)$. This can be done as follows. Each global height
  function~$h^{(j)}$ defines a quadratic function $J(F)\otimes\Q_p\to\Q_p$
  by construction ($h^{(j)}(x)=h^{\chi^{(j)}}([x-b],E_{Z_j}([x-b])+c_{Z_j})$), while each linear Chabauty function~$\int_b\omega$ defines a linear function on~$J(F)\otimes\Q_p$. If~$dg+(r_{\NS}-1)(r_2+1)>r$, then there must be an algebraic dependence between these functions on~$J(F)\otimes\Q_p$, say
\[
  P\left(\int_b\omega_1,\dots,\int_b\omega_{dg},h^{(1)},\dots,h^{(m)}\right) = 0,
\]
with~$P$ a polynomial in~$dg+m$ variables. This implies that the Coleman function
\[
  \prod_{y_1\in\Upsilon^{(1)}}\dots\prod_{y_m\in\Upsilon^{(m)}}P\left(\int_b\omega_1,\dots,\int_b\omega_{dg},y_1+\sum_{\fp\mid
  p}h_{\fp}^{(1)},\dots,y_m+\sum_{\fp\mid
  p}h_{\fp}^{(m)}\right)
\]
vanishes when evaluated on~$X(F)$.

In this way, we obtain a Coleman function~$X(F\otimes\Q_p)\to\Q_p$ that
  vanishes on~$X(F)$ for each algebraic relation between the functions
  $\int_b\omega_i$ and $h^{(j)}$ on~$J(F)\otimes\Q_p$. Since we need at least $d$ such functions to cut out a finite locus, we might try to apply this to compute~$X(F)$ when
\begin{equation}\label{Eqn : QC inequality}
    r\leq d(g-1)+(r_{\NS}-1)(r_2+1)\,.
\end{equation}

We further remark that in the case that the logarithm map~\eqref{log2} is
  \emph{injective}, then it is not too hard to write down the algebraic
  relations~$P$, and they are again either linear or quadratic relations
  among the $\int_b\omega_i$ and~$\sum_{\fp\mid p}h^{(j)}_\fp$. One such example was worked out concretely in \cite[Example~7.2]{BBBMnumberfields}. If~\eqref{log2} fails to be injective, then it is much harder to write down the algebraic relations~$P$. One approach would be to construct them via the theory of resultants, as was done over $F=\Q$ in \cite[Proposition~5.9]{balakrishnandograQC1}. Even over~$F=\Q$, additional work would be needed to turn this into an explicit method for computing $X(\Q)$: 
\cite[Proposition 5.9]{balakrishnandograQC1} has not been made explicit.

Another generalisation of~\cite{balakrishnandograQC1} is discussed by Balakrishnan and Dogra in \cite{balakrishnandograQC2}, using their theory of
\emph{generalised heights} on Selmer varieties. It would be interesting to
use these in our setting as well.

Finally, we mention two alternative approaches to quadratic Chabauty that do not use
Nekov\'a\v{r} heights. The geometric quadratic Chabauty method due to
Edixhoven and Lido~\cite{EL21} uses Poincar\'e torsors and has been
extended to the number field setting in~\cite{CLXY23} where the same
bound~\eqref{Eqn : QC inequality} is assumed. An algorithmic version
has been developed over $\Q$ in~\cite{DRHS23}, but, to our knowledge, not over
more general number fields. 
Another approach using Besser's $p$-adic Arakelov theory to construct $p$-adic heights has been developed by Besser,
M\"uller, and Srinivasan~\cite{BMSadelic}. The approach there is discussed in
detail over $\Q$, but readily extends to number fields. Again, the
bound~\eqref{Eqn : QC inequality} is a necessary condition, but at present
no algorithmic version is available.
We chose to extend the approach via Nekov\'a\v{r} heights rather than these
alternative approaches, since it seemed most straightforward to extend the
existing implementation over $\Q$.

\section{On 3-adic Galois images}\label{sec3}
We now describe how to apply the quadratic Chabauty algorithms from the previous section to a particular smooth plane quartic defined over $\Q(\zeta_3)$. Our {\tt Magma} implementation of these algorithms is available at \cite{QCMod-NF}.

\subsection{From \texorpdfstring{$X_{\ns}^+(27)$}{X\_ns+(27)} to a smooth plane quartic curve}

The classification of $\ell$-adic images for elliptic curves over~$\Q$ is complete for $\ell=2,13,17$ due to the work of several authors, including \cite{dzbrouse2adic, BDMTV_split, rszbadic, balakrishnan_dogra_müller_tuitman_vonk_2023}.

For $\ell=3$, the classification of Rouse--Sutherland--Zureick-Brown
\cite{rszbadic} is almost complete; the only case left is the occurrence of
the \emph{normaliser of the non-split Cartan subgroup}~\footnote{All such groups are conjugate, hence the associated modular
curves are isomorphic.} of level
$27$, which is the subgroup of~$\GL_2(\Z_3)$ that is the preimage of the subgroup

\[N_s(27)\coloneqq \Big\langle\begin{pmatrix}
    20 & 14\\
    7 & 20
\end{pmatrix},
\begin{pmatrix}
    2 & 9\\
    9 & 25
\end{pmatrix}\Big\rangle \leq \GL_2(\Z/27\Z)\,.\]

The corresponding modular curve~$X_{\ns}^+(27)$ has genus~$12$, and computationally, it seems infeasible to carry out quadratic Chabauty on this curve. However, in \cite[\S9.1]{rszbadic}, the authors construct a genus~$3$ quotient of~$X_{\ns}^+(27)$ defined over~$F=\Q(\zeta_3)$, which is more amenable to computation. We give two descriptions of this quotient curve. First, consider the subgroup
\[H\coloneqq \Big\langle\begin{pmatrix}
    0 & 26\\
    4 & 6
\end{pmatrix},
\begin{pmatrix}
    10 & 1\\
    25 & 26
\end{pmatrix}\Big\rangle \leq \GL_2(\Z/27\Z)\,.\]
Let~$X_H$ denote the corresponding modular curve, using the conventions in \cite{rszbadic}. The determinant~$\det(H)\subset(\Z/27\Z)^\times$ is the subgroup of elements congruent to~$1$ modulo~$3$, and so the $\Q$-scheme~$X_H$ has the structure of a geometrically connected curve over the field~$F=\Q(\zeta_3)$ (see Appendix~\ref{appx:definitions} for a precise discussion of how to view~$X_H$ as a curve over~$F$, or see \cite[Lemma~A.2]{rszbadic}).

\begin{lemma}
    $X_H$ is a quotient of the curve~$X_{\ns}^+(27)_F$.
\end{lemma}
\begin{proof}
    Let
    \[D = \{g \in \GL_2(\Z_3):\det(g) \equiv 1 \pmod{3} \}.\]
    We have
    \[
    N_s(27)\cap D \subset H
    \]
    as subgroups of~$\GL_2(\Z/27\Z)$ (see \cite[\S9.1]{rszbadic}), and so there are morphisms of modular curves
    \[
    X_{\ns}^+(27) \coloneqq X_{N_s(27)} \leftarrow X_{N_s(27)\cap D} \to X_H.
    \]
    The determinant of~$N_s(27)\cap D$ is the group of elements of~$\Z/27\Z$ congruent to~$1$ modulo~$3$, and so~$X_{N_s(27)\cap D}$ also has the structure of a curve over~$F$ and the map $X_{N_s(27)\cap D} \to X_H$ is a morphism of $F$-curves.

    Because~$N_s(27)\cap D$ has index~$2$ in~$N_s(27)$, the other projection $X_{N_s(27)\cap D}\to X_{\ns}^+(27)$ is a morphism of $\Q$-schemes which is a finite \'etale covering of degree~$2$. Together with the structure map to~$\Spec(F)$, this projection induces (by the universal property of products) a map $X_{N_s(27)\cap D}\to X_{\ns}^+(27)_F=X_{\ns}^+(27)\times\Spec(F)$ which is a finite \'etale morphism of degree~$1$, i.e.~an isomorphism. So the morphism $X_{N_s(27)\cap D}\to X_H$ above exhibits~$X_H$ as a quotient of~$X_{\ns}^+(27)$ of degree~$3=[N_s(27)\cap D:H]$.
\end{proof}

In \cite{rszbadic}, the authors use a slightly different construction of a quotient of~$X_{\ns}^+(27)_F$, which only uses the structure of~$X_H$ as a $\Q$-scheme. If we let~$X_{H,F}\coloneqq X_H\times_{\Spec(\Q)}\Spec(F)$, then~$X_{H,F}$
is the disjoint union of two geometrically connected $F$-curves. These two
components are conjugate to one another under the action of $\Gal(F/\Q)$
(not, as claimed in \cite[\S9.1]{rszbadic}, isomorphic to one another, as
one can verify by computing their Dixmier--Ohno invariants). 

The authors of \cite{rszbadic} compute an equation for one component~$X_H'$ of~$X_{H,F}$: it is the plane quartic curve given by
\begin{equation}
\label{Eqn: plane quartic}
\begin{gathered}
 f(a,b,c)\coloneqq  a^4 + (\zeta_3-1)a^3b + (3\zeta_3 + 2)a^3c-3a^2c^2 + (2\zeta_3 + 2)ab^3 - 3\zeta_3ab^2c \\
    \qquad\qquad \quad \mathop{+} 3\zeta_3abc^2 - 2\zeta_3ac^3 - \zeta_3 b^3c + 3\zeta_3 b^2 c^2 + (-\zeta_3 + 1)bc^3 + (\zeta_3 + 1)c^4=0\,.
\end{gathered}
\end{equation}

\begin{rmk}\label{Remark :XHvsXH'}
The two components of~$X_{H,F}$ are~$X_H$ and its conjugate. The conjugate of~$X_H$ is $F$-isomorphic to the modular curve associated to the subgroup
\[
\tilde H \coloneqq \begin{pmatrix}0&1\\1&0\end{pmatrix}\cdot H\cdot \begin{pmatrix}0&1\\1&0\end{pmatrix} \,,
\]
see Corollary~\ref{cor:base_change} in the appendix.

Thus~$X_H'$ is $F$-isomorphic to either~$X_H$ or~$X_{\tilde H}$, and either way it is a modular curve. At the time of writing, we do not know which of~$X_H$ or~$X_{\tilde H}$ is isomorphic to~$X_H'$.
\end{rmk}

Henceforth, we write $X$ for the curve $X_H'$. 
It has thirteen known $F$-rational points.  
Let $(x,y)$ denote the point $[a,b,c]=[x,y,1]$ on the curve determined by \eqref{Eqn: plane quartic}, and let $\infty_1=\left[1,0,0\right], \infty_2=\left[1,\zeta_3+1,0\right]$. Then the known points are 
\begin{equation} \label{listpts}
\begin{gathered}
      \biggl\{\left( 0,-\zeta_3-1 \right), \left( 1,-\zeta_3-1 \right),          \left( \zeta_3+1,-\zeta_3-1 \right),\left( 0,-\zeta_3 \right),  \biggr. \\
 \left.    
      \left( \zeta_3+1,0 \right), \left( 2\zeta_3+2,\zeta_3 \right),\left( \zeta_3 ,1\right),     
  \left( \frac{\zeta_3-3}{2},\frac{\zeta_3+2}{2} \right), \right. \\
  \biggl.
 \left( \frac{-\zeta_3-2}{3},\frac{\zeta_3+2}{3}  \right), \left( \frac{-\zeta_3}{2},\frac{-1}{2} \right), \left( \frac{5\zeta_3+4}{7}, -1 \right),
  \infty_1,\infty_2 \biggr\}.
  \end{gathered}
\end{equation}

The point \[ \left(\frac{\zeta_3-3}{2},\frac{\zeta_3+2}{2}
\right)\] is the only non-CM point in this list. 
The corresponding $j$-invariant is
\begin{equation*}\label{eqn: jinv}
    j_0\coloneqq
2^3 \cdot 5^3 \cdot 19^{-27} \cdot (1-\zeta_3)^6 \cdot (2-3\zeta_3)^{27}
  \cdot (27+13 \zeta_3)^3\cdot (54+49\zeta_3)^3\cdot (227+173 \zeta_3)^3\,.
\end{equation*}

\begin{rmk}
    The factorisation of the $j$-invariant $j_0$ given above differs from the one claimed in \cite[\S9.1]{rszbadic}. In fact, the latter does not lift to an $F$-point on $X_{\ns}^+(9)$. 

\end{rmk}

Our goal for the rest of this section is to show that $X(F)$ is exactly equal to the set of the thirteen points in~\eqref{listpts}.

\subsection {Arithmetic invariants of
\texorpdfstring{$X$}{X}}\label{S:invs}
We first discuss why the curve $X/F$ is a suitable candidate for quadratic Chabauty by computing some invariants of the curve and its Jacobian.

    The model $f=0$ of the curve as in~\eqref{Eqn: plane quartic} has bad reduction only at the unique prime above $3$. Indeed, 
    the norm of the discriminant of $f$ is $3^{54}$.

     The abelian variety $A\coloneqq \Res^F_{\Q}(J)$ is of $\GL_2$-type. By the proof of Serre's conjecture by Khare--Wintenberger \cite{khare2009serre}, this implies that $A$ is modular; that is, its $L$-function is the same as the $L$-function of a weight $2$ modular form.  By~\cite[\S9.1]{rszbadic}, the newform orbit associated to $A$ is the Galois orbit of the newform \href{https://www.lmfdb.org/ModularForm/GL2/Q/holomorphic/729/2/a/c/}{729.2.a.c},  
    and any newform in this orbit has analytic rank $1$.
    Since the Hecke algebra is isomorphic to $\Q(\zeta_{36})^+$, the
    maximal real subfield of $\Q(\zeta_{36})$, which has degree $6$, 
    the analytic rank of $A$ is $6$. By work of Gross--Zagier \cite{gross1986heegner} and Kolyvagin--Logachev \cite[Theorem 0.3]{kolyvaginlogachev}, we get \[\rk_{\Z} J(F)=\rk_{\Z}A(\Q)=6.\]
     Moreover, the form 
     \href{https://www.lmfdb.org/ModularForm/GL2/Q/holomorphic/729/2/a/c/}{729.2.a.c}  
     has exactly one non trivial inner twist.  Hence by work of Pyle \cite{Pyle}, $\End^0(J)$ is isomorphic to a subfield $K\subseteq \Q(\zeta_{36})^+$ such that $[\Q(\zeta_{36})^+:K]=2$. Thus we deduce $K=\Q(\zeta_9)^+$. One can alternatively show this using the algorithms in \cite{CMSV_endo}.

As a consequence, we may apply quadratic Chabauty to $X/F$, since we have
\begin{equation}\label{invariantsX}
g=3, \;d=2, \;r=6,\; r_{\NS}=3\,\; \text{and}\; r_2=1\,.
\end{equation}

\subsection{Set up}

To implement our algorithm to compute quadratic Chabauty pairs~$(\rho,\Upsilon)$ for the curve~$X$, we require the following inputs:

\begin{enumerate}
\item An appropriately chosen rational prime $p$.
\item Two independent id\`ele class characters~$\chi^{(1)},\chi^{(2)}$ on $F$.

\item A polynomial $Q(x,y)$ whose vanishing locus is birational to an open
  affine $Y\subseteq X$. We require this polynomial $Q$ to satisfy
    additional conditions which are discussed in~\S\ref{Subsection: Affine patch}. We provide a suitable choice of $Q$ in \eqref{Eqn: Q}.

\item Differentials $\omega_0,\ldots,\omega_{2g-2+\delta}$  as in
  \eqref{eq:211}--\eqref{eq:214}. 
\item Splittings $s_{\fp_1}, s_{\fp_2}$ of the Hodge filtration on $H^1_{\dR}(X_{F_{\fp_i}})$ satisfying Lemma~\ref{L:equivsplit}.

\item The action of two independent cycles $Z_1,Z_2\in   \ker(\NS(J)\to \NS(X))$ on $H^1_{\dR}(X)$ with respect
  to the basis $\vecc{\omega}$. See~\S\ref{Subsection: Correspondences}.

\item The finitely many values the local heights $h_{\mfq}$ for $\mfq \nmid p$ can take. We show these values are $0$ for all $\mfq\nmid p$ for $X$ in~\S\ref{Sec: Hts for X away}.
\end{enumerate}
 
 Recall that our rational prime needs to be split in $F$, so we need $p\equiv
 1\pmod 3$. We also require all $\fp\mid p$ to satisfy
 Tuitman's Assumption 1 \cite{tuitman2017counting}. We choose $p=13$, which is the
 smallest prime satisfying these conditions, and we
let $\fp_1,\fp_2$ be the primes of $F$ above~13.

 The space of $\Q_p$-valued continuous id\`ele class characters on $F$ is
 two-dimensional. We pick the characters $\chi^{(1)},\chi^{(2)}$ such that
 $\chi^{(1)}_{\fp_1}=\log_{\fp_1}, \chi^{(1)}_{\fp_2}=0$ and
 $\chi^{(2)}_{\fp_2}=\log_{\fp_2}, \chi^{(2)}_{\fp_1}=0$, where we choose
 the logarithms that vanish
 at uniformisers. See~\cite[\S2.1]{BBBMnumberfields} for a discussion of
 explicit computations with id\`ele class characters.

\subsection{A suitable affine patch} \label{Subsection: Affine patch}

We require $Q\in F[x,y]$ to satisfy the following conditions:
\begin{enumerate}
\item For $\fp\mid p$, the polynomial $Q$ has $\fp$-integral coefficients, is monic in $y$ and satisfies the conditions in \cite[Assumption 1]{tuitman2016counting}. \label{Condition: (1) for Q}

\item All residue polydiscs in $X(F\otimes\Q_p)$ are good with
  respect to the affine patch $Y$ defined by $Q(x,y)=0$.\label{Condition: (2) for Q}
\end{enumerate}

We need a $Q$ that satisfies~\eqref{Condition: (1) for Q} to run Tuitman's algorithm, as
discussed in~\S\ref{subsec:Hodge} and~\S\ref{subsec:Frob}. Condition \eqref{Condition: (2) for Q} was not necessary in previous
implementations of quadratic Chabauty over $\Q$
\cite{balakrishnan_dogra_müller_tuitman_vonk_2023}, where one could work
with two affine opens $Y_1, Y_2\subseteq X$ such that every residue disc in $X(\Q_p)$ is good with respect to $Y_1$ or $Y_2$.

Over $F=\Q(\zeta_3)$, if one chooses two affine opens $Y_1,Y_2$ covering
$X$ such that neither satisfies Condition~\eqref{Condition: (2) for Q}, one
also needs to deal with residue discs of the form $D_1\times D_2$, where
$D_i$  is a good residue disc of the affine patch
$Y_{i,{F_{\fp_i}}}$. The current
implementation of our algorithm only lets us compute approximations of
Coleman functions on products of residue discs $D_1\times D_2\subseteq
Y_{F_{\fp_1}}\times Y_{{F_{\fp_2}}}$, for a choice of affine open
$Y\subseteq X$. Therefore, Condition \eqref{Condition: (2) for Q} ensures that we can use just one patch to account for all residue discs on $X(F\otimes\Q_p)$. 
The choice of $Q$ we use satisfying the required conditions is
\[\label{Eqn: Q}
\begin{gathered}
  Q(x,y) =y^4 + ((-2\zeta_3 + 9)x + (2\zeta_3 + 3))y^3 + (-3x^2 + 6x - 3)y^2 \\
\mathop{+}\, ((-170\zeta_3 + 254)x^3 + (-150\zeta_3 + 114)x^2 + (-54\zeta_3 + 18)x - 10\zeta_3 - 2)y \\
\mathop{+}\, (162\zeta_3 + 144)x^4 + (-108\zeta_3 + 48)x^3 + (-72\zeta_3 - 144)x^2 + (12\zeta_3 - 48)x + 6\zeta_3.
\end{gathered}
\]
We compute the differentials $\omega_i$ as described in~\S\ref{subsec:Hodge}, and we apply linear algebra to find splittings $s_{\fp_1}, s_{\fp_2}$ satisfying Lemma~\ref{L:equivsplit}.

\subsection{Hecke correspondence} \label{Subsection: Correspondences}

As discussed in~\S\ref{S:invs}, the endomorphism algebra of $J$ is $K
\coloneqq \Q[x]/(x^3 - 3x - 1)=\Q(\zeta_9)^+$. We need to 
explicitly describe the action of a nontrivial element $Z\in \ker(\NS(J)\to \NS(X))$ on
$H^1_{\dR}(X)$. Recall from Remark~\ref{Remark :XHvsXH'} that $X$ is $F$-isomorphic to a modular
curve. For $X_H$, we can define Hecke operators and Hecke correspondences by
choosing coset representatives appropriately; see, for example, \cite[Lemma
4.4.14]{assafmodularforms}. These Hecke correspondences are automatically rational over the ground field~$F$ \cite[Theorem~7.9]{shimura1971introduction}. One can calculate the action of a Hecke
correspondence by using the Eichler--Shimura relation as discussed in Remark~\ref{R:Z}.
Fix a prime $\fp\mid p$. We denote by $\Frob_p$ absolute
Frobenius on $X_{O_F/\fp}$, and we let $\Frob_p^*$ be the pullback of
absolute Frobenius on rigid cohomology and, consequently, de Rham
cohomology since $H^1_{\rig}(X_{O_F/\fp})\cong H^1_{\dR}(X_{F_{\fp}})$. 

 By the Eichler--Shimura relation, the action of the Hecke operator $T_p$ on $H^1_{\dR}(X_{F_{\fp}})$ is 
\[T_p=\Frob^*_{p}+p(\Frob_{p}^*)^{-1},\] 
following \cite[Corollary 7.10]{shimura1971introduction}. 
  To use this in practice, we fix a sequence of $2g$ differentials
  $\vecc{\omega}$ as in \eqref{eq:211}-\eqref{eq:214}. Let $\Frob$ be the
  matrix describing the lift of Frobenius used by Tuitman, and let $H_\fp$ be the matrix describing the action of the Hecke correspondence on $H^1_{\dR}(X_{F_{\fp}})$, both with respect to the basis of $H^1_{\dR}(X_{F_{\fp}})$ determined by $\vecc{\omega}$. 

Let $C$ be the cup product matrix on $H^1_{\dR}(X_{F_{\fp}})\times H^1_{\dR}(X_{F_{\fp}})$ with respect to $\vecc{\omega}$. For $i=1,2$, let \begin{equation}\label{Eqn: Z}
    Z_i\coloneqq (\Tr(H_\fp^i)I_{2g}-2gH_\fp^i)C^{-1}.
\end{equation}
The matrices $Z_1,Z_2$ represent endomorphisms of $H^1_{\dR}(X_{F_{\fp}})$ that have trace zero and correspond to independent elements of
  $\ker(\NS(J)\to \NS(X))$ (see
  \cite[Section~3.5.2]{balakrishnan_dogra_müller_tuitman_vonk_2023}).

Having two independent correspondences is useful for solving for the height pairing since we get more elements $e(x)$ for fitting the height pairing as discussed in Remark \ref{Remark : More correspondences}. Moreover, by also using two id\`ele class characters, we obtain four quadratic Chabauty pairs. We expect two such pairs to suffice to give a finite set of points in $X(F\otimes\Q_p)$ containing $X(F)$, but by using the other pairs as well, we expect the common zero set to be precisely $X(F)$.

\begin{Remark}
    Since $T_p$ is actually a correspondence on $X$ over $F$, the matrix
    $H_\fp\in M_{2g\times2g}(\Q_p)$ is the $p$-adic approximation of a matrix in $M_{2g\times 2g}(F)$.
    We found it convenient to recognise the entries of
    the $\Q_p$-matrices $Z_i$ as algebraic numbers using LLL with $344$
    $p$-adic digits of precision. Although we expect that these algebraic
    numbers are exactly equal to the entries of the $Z_i$, we do not explicitly
    prove this. This, in fact, suffices for our 
     algorithms: they all work directly over $\Q_p$, except for the one
     described in~\S\ref{subsec:Hodge}, which requires the entries of $Z_i$
     to be in $F$. However, for all algorithms, it suffices to show that
     these entries are accurate up to a
    suitable explicit $\fp$-adic precision, and this can be obtained from
    the accuracy of the matrices of Frobenius, which is known from the
    analysis of Tuitman's algorithm in~\cite{tuitman2016counting,
    tuitman2017counting}.
\end{Remark}
\subsection{Local heights away from  \texorpdfstring{$p$}{p} } \label{Sec: Hts for X away}

In this section, we determine the local heights on the curve $X$ away from
$p$. We fix a choice of~$Z\in\ker(\NS(J)\to\NS(X))$ and an id\`ele class character~$\chi$, and abbreviate the local height $h_{Z,\mfq}^{\chi_{\mfq}}$ to~$h_{\mfq}$ to avoid notational clutter. We will show  
\begin{prop}\label{prop: local_heights_trivial}
    Let~$\mfq$ be a prime of~$F$, not dividing~$p$. Then $h_{\mfq}(x)=0$ for all~$x\in X(F_\mfq)$.
    \end{prop}
Since~$X$ has good reduction away from the prime $\mfq_3=(1-\zeta_3)$,
the local heights at all of these primes are certainly trivial by Lemma~\ref{Lemma: BettsDogra height}. The content of Proposition~\ref{prop: local_heights_trivial} lies then in tackling the prime~$\mfq=\mfq_3$. Henceforth, we will write~$F_3$ for~$F_{\mfq_3}=\Q_3(\zeta_3)$. By Lemma~\ref{Lemma: BettsDogra height}, it suffices to show that all $F_3$-points of $X$ reduce onto the same component of the special fibre of some regular semistable model of~$X$ over some finite extension of~$F_3$. In our case, the requisite models were computed by Ossen.
\begin{theorem}
[\protect{\cite[Theorem 1] {ossen2023computing}}]\label{Thm: Regular semistable model at 3}
There exists a finite Galois extension $L/F_3$ of ramification index $54$
  over which $X_L$ has semistable reduction. Moreover the special fibre of
  the stable model of $X_L$ consists of three genus~1 components
  $E_1,E_2,E_3$ and one rational component $Z$.  Each genus~1 component is smooth and is birational to the plane curve \begin{equation}\label{Eqn: Genus 1 curves in SMP}
           y^3-y=x^2.
       \end{equation} The irreducible components are configured as in
       Figure ~\ref{Fig: Special fibre stable}. \end{theorem}
\begin{figure}[ht]
  \centering
  \includegraphics[scale=0.25]{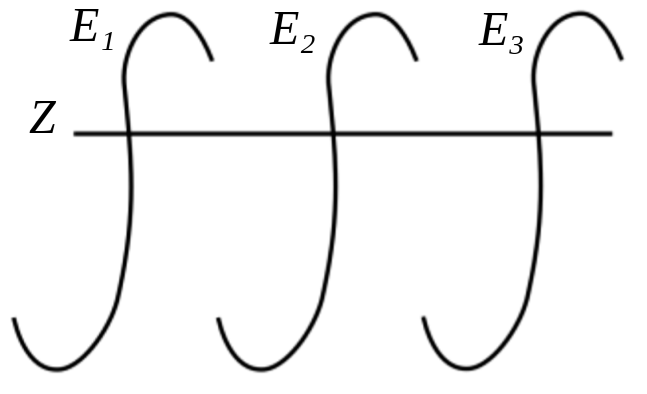}
  \caption{The special fibre of the stable model of $X_L$ \cite[Figure 2]{ossen2023computing}.}
  \label{Fig: Special fibre stable}
\end{figure}
  Ossen produces the stable model in Theorem~\ref{Thm: Regular semistable model at 3} by using a degree~$3$ map $\phi\colon X_{F_3} \to \P^1_{F_3}$ and taking the normalisation of a suitable semistable model~$\cY$ of~$\P^1_{\bar F_3}$ in~$X_{\bar F_3}$. To produce the model~$\cY$, Ossen considers the degree~$9$ polynomial
\begin{align}\label{Eqn: discoid}
\begin{split}
\psi=  x^9 + (9\zeta_3 - 9)x^8 + (54\zeta_3 + 27)x^7 + (54\zeta_3 - 27/2)x^6 + (243\zeta_3 + 972)x^5 + 729\zeta_3x^4\\
\mathop{+}\, (2916\zeta_3 - 1458)x^3 + (37179\zeta_3 + 41553)x^2 + (6561\zeta_3 + 6561/8)x - 63423\zeta_3 + 155277,
\end{split}
\end{align}
which is irreducible over~$F_3$. For each root~$\alpha$ of~$\psi$, there are exactly two other roots a distance $3^{-3/2}$ from~$\alpha$ ($3$-adically), and the remaining six roots lie at a distance $3^{-7/6}$ from~$\alpha$.

Let~$D$ be the closed $3$-adic disc of radius $3^{-7/6}$ containing all nine roots of~$\psi$, and let~$D_1,D_2,D_3$ be the three discs of radius $3^{-17/12}$ containing three roots of $\psi$ each. Corresponding to the configuration of discs $\{D,D_1,D_2,D_3\}$ is a semistable model~$\cY$ of $\P^1_{\bar F_3}$, as described in \cite[\S1]{ossen2023computing} or \cite[Corollary~3.18 \& Theorem~4.56]{ruth2015models}. We can construct~$\cY$ explicitly as follows. Let~$\alpha$, $\alpha_1$, $\alpha_2$, and $\alpha_3$ be centres of the discs $D$, $D_1$, $D_2$, and $D_3$, respectively. Then~$\cY$ is the projective closure of the blowup of $\Spec(O_{\bar F_3}[3^{-7/6}(x-\alpha)])$ at the ideals $(3^{-11/12}(x-\alpha),3^{-7/6}(x-\alpha_i))$ for~$i=1,2,3$. The special fibre of~$\cY$ is then the union of four components $Y,Y_1,Y_2,Y_3$ arranged as in Figure~\ref{F:Y}.
\begin{figure}[ht]
\includegraphics[scale=0.25]{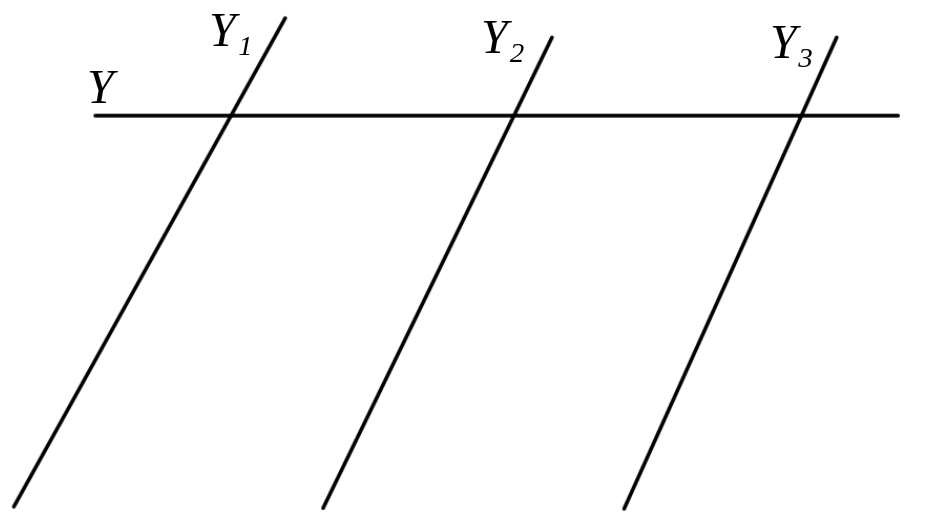}
\caption{The special fibre of~$\cY$.}\label{F:Y}
\end{figure}

Ossen shows that the normalisation of~$\cY$ in~$X_{\bar F_3}$ is a stable model~$\cX$ of~$X_{\bar F_3}$, with special fibre as depicted in Figure~\ref{Fig: Special fibre stable}, and the induced map $\cX_s\to\cY_s$ on special fibres maps the $\P^1$ component of~$\cX_s$ surjectively to~$Y$, and the three genus one components surjectively to~$Y_1,Y_2,Y_3$, respectively.
The Galois group of~$F_3$ acts on~$\cX$ (semilinearly over~$O_{\bar F_3}$, i.e.~so that the map $\cX\to\Spec(O_{\bar F_3})$ is Galois-equivariant), and thus on the special fibre. We have

\begin{lemma}\label{Lemma: transitive action}    $\Gal(L/F_3)$ acts transitively on the three genus one components of~$\cX_s$.
\end{lemma}
\begin{proof}
  Let~$E_i$ denote the genus one component living over~$Y_i$. 
  Let~$\alpha\in\P^1(\bar F_3)$ be a root of~$\psi$ inside the disc~$D_1$, so the reduction of~$\alpha$ in the model~$\cY$ lies on the component~$Y_1$. Since~$\psi$ is irreducible, there is an element~$\sigma\in G_{F_3}$ such that~$\sigma(\alpha)$ lies in the disc~$D_2$, and so reduces onto~$Y_2$.

    Choose a lift~$\alpha'\in\phi^{-1}(\alpha)$ to~$X(\bar F_3)$. The reduction of~$\alpha'$ in the model~$\cX$ must lie on the component~$E_1$, since it maps to a point in $Y_1\subset\cY_s$. The point~$\sigma(\alpha')$ is a lift of~$\sigma(\alpha)$, so reduces onto the component~$E_2$ similarly. Since the reduction map $X(\bar F_3) \to \cX(\bar\F_3)$ is Galois-equivariant, this implies that~$\sigma(E_1)=E_2$. Since our labelling of the components~$E_1$, $E_2$, and~$E_3$ was arbitrary, this implies that~$G_{F_3}$ acts transitively.
\end{proof}

As a consequence of Lemma~\ref{Lemma: transitive action}, it follows that the reduction of any $F_3$-rational point of~$X$ must be a smooth point of~$\cX_s$ lying on the genus~$0$ component (since the reduction of this point is Galois-invariant).

Now let~$L/F_3$ be a finite extension over which~$X$ has semistable reduction, and by mild abuse of notation let~$\cX$ also denote the stable model over~$O_L$. We may suppose without loss of generality that all four geometric components of~$\cX_s$ are defined over the residue field of~$L$. By \cite[Theorem~10.3.34]{liu:algebraic_geometry}, the minimal regular model~$\cX^{\min}$ of~$X_L$ is semistable, and the stable model~$\cX$ is formed by contracting all components of the special fibre of~$\cX^{\min}$ of self-intersection~$-1$, which become singular points in~$\cX$. In particular, $\cX^{\min}$ dominates~$\cX$ and the map~$\cX^{\min}\to\cX$ is an isomorphism away from the singular points of~$\cX_s$. Thus, we obtain
\begin{lemma}
    All $F_3$-rational points on~$X$ reduce onto the same component of the special fibre of~$\cX^{\min}$, namely the unique component dominating the genus~$0$ component of~$\cX_s$.
\end{lemma}

This implies, via Lemma~\ref{Lemma: BettsDogra height}, that all
$F_3$-rational points on~$X$ have local height~$0$, and thus completes the proof
of Proposition~\ref{prop: local_heights_trivial}. \qed

\subsection{Local computations above \texorpdfstring{$p$}{p} and root finding}

Using Algorithm \ref{Algo: Hodge filtration} and Algorithm \ref{Algorithm: Frobenius structure}, we expand, for $i\in \{1,2\}$, the local heights $h^{\chi^{(1)}}_{\fp_1,Z_i}$ and $h^{\chi^{(2)}}_{\fp_2,Z_i}$ in good residue discs. We can also compute them at local points and use this to compute 
$h^{\chi^{(1)}}_{Z_i}=h^{\chi^{(1)}}_{\fp_1,Z_i}$ and $h^{\chi^{(2)}}_{Z_i}=h^{\chi^{(2)}}_{\fp_2,Z_i}$ 
at all $F$-points, since we chose our affine patch so that they all map to good residue polydiscs. The same algorithms allow us to compute $e_{Z_1}(x)$ and $e_{Z_2}(x)$ for all known $F$-points $x$, and we find that these suffice to generate $T\otimes_{K\otimes\Q_p} T$. This allows us to write the height pairings $h^{\chi^{(1)}}$ and $h^{\chi^{(2)}}$ in terms of the corresponding dual basis of  $(T\otimes_{K\otimes\Q_p} T)^*$ and to extend $h^{\chi^{(1)}}_{Z_i}$ and $h^{\chi^{(2)}}_{Z_i}$ to Coleman functions on good residue polydiscs as discussed in~\S\ref{subsec:solving}. 

Using this, we find expansions of   the functions 
$$\rho_{i,j} \coloneqq 
h^{\chi^{(j)}}_{Z_i} - h^{\chi^{(j)}}_{\fp_j,Z_i}\,,\quad i,j \in \{1,2\} 
$$
on all good residue polydiscs. They
are non-constant and therefore yield four quadratic Chabauty pairs $(\rho_{ij},
\{0\})$. 
We then solve for the common zeroes of $\rho_{1,1}$ and $\rho_{1,2}$ using
the method outlined in~\S\ref{subsec:Hensel}. There are finitely many common
zeroes, and they all satisfy the
multivariate Hensel condition discussed in~\S\ref{subsec:prechensel}. We
then verify that the functions $\rho_{2,1}$ and
$\rho_{2,2}$ also vanish at these solutions.
Hence, by the discussion in~\S\ref{subsec:prechensel}, we provably find
all common zeroes of the $\rho_{i,j}$ (to precision $13^5$). These coincide with the images of the known $F$-points listed in~\eqref{listpts}.

\subsection{Finishing the proofs of Theorems~\ref{thm:XH'F}, \ref{thm:ns27F} and \ref{thm:ns27}}\label{S:proofs}
Using the results of Section~\ref{Section : QC for modular curves}, the computations described in the present section combine to prove Theorem~\ref{thm:XH'F}.

If one can compute the degree $3$ map \[X_{\ns}^+(27)_F\to X_H,\]
which is induced by an inclusion of subgroups $N_{\ns}(27)\cap D\subseteq H$, 
one can obtain the set $X_{\ns}^+(27)(F)$ and hence the set $X_{\ns}^+(27)(\Q)$ by pulling back the set $X(F)$, which by Theorem \ref{thm:XH'F} is equal to the points listed in~\eqref{listpts}.

Jeremy Rouse provided us with equations of the canonical model of $X_{\ns}^+(27)$ in $\P^{11}$ along with explicit formulas for rational maps representing quotient morphisms
\begin{align*}
   f_1\colon X_{\ns}^+(27)_F&\to X, \\
f_2\colon X_{\ns}^+(27) &\to X_{\ns}^+(9),\\
f_3\colon X &\to X_{\ns}^+(9)_F\,.
\end{align*} 

  The rational maps representing $f_2$ and $f_3$ are defined everywhere, while the map representing $f_1$ is not defined on one of the $F$-points of $X_{\ns}^+(27)$. Thus, we compute $X_{\ns}^+(27)(F)$ by identifying the $F$-points of $f_2^{-1}(f_3(X(F)))$, instead of identifying the $F$-points of $f_1^{-1}(X(F))$.

We find that the curve $X_{\ns}^+(27)$ has exactly eight $\Q$-points, and these correspond to the $8$ CM discriminants $-4,-7,-16,-19,-28,-43,-67,-163$. This proves Theorem~\ref{thm:ns27}.

In addition, we find two points in~$X_{\ns}^+(27)(F)\setminus X_{\ns}^+(27)(\Q)$. These also correspond to CM elliptic curves with CM discriminant $-4$. This proves Theorem~\ref{thm:ns27F}.

\appendix

\section{\for{toc}{Comparing definitions of modular curves}\except{toc}{Comparing definitions of modular curves\texorpdfstring{\\}{ }\textup{by L.~Alexander Betts}}}\label{appx:definitions}

In the body of this paper, we rely on works using two perspectives on modular curves: the paper of Rouse, Sutherland and Zureick-Brown \cite{rszbadic} defines modular curves as solutions to a particular moduli problem, while Shimura's book \cite{shimura1971introduction} defines modular curves as quotients of the upper half plane~$\H$, with an explicit algebraic structure. The purpose of this appendix is to compare these two definitions of the modular curve attached to an open subgroup~$G\subseteq\GL_2(\hat\Z)$, so that we can use both~\cite{rszbadic} and~\cite{shimura1971introduction} simultaneously. Ultimately, we will find that the definitions are not actually the same, but differ by transposing the subgroup~$G$ (which is not a problem for us).

Throughout this appendix, we will pay particular attention to the fields of definition of modular curves. That is, the modular curve~$X_G$ is initially defined as a $\Q$-scheme, but as we will explain (following \cite[Lemma~A.2]{rszbadic}), it actually has the structure of a curve over a number field~$F_G$ depending on~$G$. When we compare the moduli-theoretic and holomorphic definitions of~$X_G$, we will take care to track the base fields throughout and any necessary twists which may occur.

\subsection{The moduli-theoretic definition}

First, we recall the moduli-theoretic definition of modular curves, following the conventions of~\cite{rszbadic}. If~$E$ is an elliptic curve over a base $\Q$-algebra~$R$, then by a \emph{full level~$N$ structure} on~$E$, we mean a pair of elements~$P,Q\in E[N](R)$ which form a $\Z/N\Z$-basis of~$E[N]$ at every geometric point of~$\Spec(R)$. Equivalently, we can think of a full level~$N$ structure as being an isomorphism
\[
\iota\colon E[N] \xrightarrow\sim (\Z/N\Z)^{\oplus2})_R
\]
of \'etale group-schemes over~$\Spec(R)$, namely the isomorphism determined by~$\iota(P)=e_1$ and~$\iota(Q)=e_2$ with~$(e_1,e_2)$ the standard basis of~$(\Z/N\Z)^{\oplus2}$.

If~$N\geq3$, then the moduli problem classifying elliptic curves with full level~$N$ structure has a fine moduli space over~$\Q$, which we denote by~$Y(N)$. The moduli space~$Y(N)$ is a one-dimensional smooth reduced affine scheme of finite type over~$\Q$ and, as we will see, is geometrically disconnected over~$\Q$. We refer to~$Y(N)$ as the \emph{full affine modular curve of level~$N$} and write~$X(N)$ for its smooth compactification.

As defined, $Y(N)$ is a geometrically disconnected curve over~$\Q$, but it is often useful to instead view it as a curve over a different number field. For this, recall that if~$F$ is a field, then a scheme defined over~$F$ is the same thing as an abstract scheme~$X$ together with a scheme morphism
\[
X \to \Spec(F) \,.
\]
(See, for example, \cite[\S7.3.7]{vakil:rising_sea}.) Morphisms defined over~$F$ are then the same thing as scheme morphisms compatible with the structure maps to~$F$. From this perspective, one can see that~$Y(N)$ admits the structure of a smooth reduced curve over the number field~$\Q(\zeta_N)\coloneqq \Q[T]/(\Phi_N(T))$. (The element $\zeta_N\coloneqq T$ mod~$\Phi_N(T)$ is thus a primitive $N$th root of unity in~$\Q(\zeta_N)$.) Let~$E_N$ be the universal elliptic curve over~$Y(N)$, which comes with two sections~$P_N,Q_N$ of the projection $E_N\to Y(N)$ giving a trivialisation of the $N$-torsion subscheme~$E_N[N]$. The mod~$N$ Weil pairing of~$P_N$ and~$Q_N$ defines a primitive $N$th root of unity in~$\cO(Y(N))$ (meaning a root of the $N$th cyclotomic polynomial~$\Phi_N$), and hence a morphism
\[
Y(N) \to \Spec(\Q[T]/(\Phi_N(T))) = \Spec(\Q(\zeta_N))
\]
of $\Q$-schemes. This morphism endows~$Y(N)$ with the structure of a scheme over~$\Q(\zeta_N)$. Inside $\cO(Y(N))$, the identity $e_N(P_N,Q_N)=\zeta_N$ holds, and this even defines the $\Q(\zeta_N)$-algebra structure on~$\cO(Y(N))$.

Viewed as a $\Q(\zeta_N)$-scheme, $Y(N)$ admits the following moduli interpretation.

\begin{lemma}\label{lem:new_moduli_description}
    $Y(N)$, viewed as a $\Q(\zeta_N)$-scheme, is the moduli space of elliptic curves~$E$ over $\Q(\zeta_N)$ equipped with a basis~$(P,Q)$ of~$E[N]$ such that~$e_N(P,Q)=\zeta_N$.
\end{lemma}
\begin{proof}
    The assertion to be proved is that for any $\Q(\zeta_N)$-algebra~$R$, any elliptic curve~$E$ over~$R$ and any basis~$(P,Q)$ of~$E[N]$ such that~$e_N(P,Q)=\zeta_N$, there is a unique morphism $f\colon\Spec(R)\to Y(N)$ of $\Q(\zeta_N)$-schemes such that~$(E,P,Q)$ is the pullback of the universal elliptic curve with $N$-level structure~$(E_N,P_N,Q_N)$ over~$Y(N)$. The moduli interpretation for~$Y(N)$ as a $\Q$-scheme provides a morphism~$f\colon \Spec(R)\to Y(N)$ such that~$(E,P,Q)$ is the pullback of~$(E_N,P_N,Q_N)$, but~$f$ is initially only a morphism of~$\Q$-schemes.
    
    The map~$f$ is a morphism of~$\Q(\zeta_N)$-schemes if and only if it is compatible with the structure maps to~$\Spec(\Q(\zeta_N))$, i.e.~if and only if the composition
    \[
    \Spec(R) \xrightarrow{f} Y(N) \to \Spec(\Q(\zeta_N))
    \]
    is equal to the structure map $\Spec(R)\to\Spec(\Q(\zeta_N))$ coming from the fact that~$R$ is a $\Q(\zeta_N)$-algebra. This composition sends the element~$\zeta_N\in\Spec(\Q(\zeta_N))$ to $f^*e_N(P_N,Q_N)=e_N(P,Q)$, and so~$f$ is a morphism of $\Q(\zeta_N)$-schemes if and only if the trivialisation~$(P,Q)$ satisfies~$e_N(P,Q)=\zeta_N$. The result follows.
\end{proof}

One can also define modular curves associated to arbitrary open subgroups of~$\GL_2(\hat\Z)$. There is a natural left action of the finite group~$\GL_2(\Z/N\Z)$ on~$Y(N)$ given by
\[
\gamma\cdot(E,\iota) \coloneqq (E,\gamma\circ\iota)
\]
for~$\gamma\in\GL_2(\Z/N\Z)=\Aut((\Z/N\Z)^{\oplus2})$, and~$\iota\in\operatorname{Iso}(E[N],(\Z/N\Z)^{\oplus2})$. In terms of bases~$(P,Q)$ of~$E[N]$, this action is given by
\begin{equation}\label{eq:left_action}
    \begin{pmatrix}a&b\\c&d\end{pmatrix}\cdot(E,P,Q) \coloneqq (E,a'P+c'Q,b'P+d'Q) \,,
\end{equation}
where
\begin{equation}\label{eq:inverse_matrix}
    \begin{pmatrix}a'&b'\\c'&d'\end{pmatrix} = \frac1{ad-bc}\begin{pmatrix}d&-b\\-c&a\end{pmatrix} = \begin{pmatrix}a&b\\c&d\end{pmatrix}^{-1} \,.
\end{equation}
This action is an action by automorphisms of~$\Q$-schemes; on the level of~$\Q(\zeta_N)$-schemes, it is a semilinear action with respect to the action on~$\Q(\zeta_N)$ given by $\gamma\cdot\zeta_N=\zeta_N^{\det(\gamma)^{-1}}$ (this follows from the identity $\zeta_N=e_N(P_N,Q_N)$). If~$G$ is an open subgroup of~$\GL_2(\hat\Z)$ of level~$N\geq3$, we define the coarse modular curve of level~$G$ to be the coarse quotient $X_G=X(N)/(\text{$G$ mod $N$})$. The structure map $X(N)\to\Spec(\Q(\zeta_N))$ induces, on quotienting by~$G$ mod~$N$, a morphism $X_G \to \Spec(F_G)$, where~$F_G\subseteq \Q(\zeta_N)$ is the fixed field of~$\det(G)$ mod~$N$. This makes~$X_G$ into a curve over the number field~$F_G$.

\subsubsection{Functoriality}

The construction of the modular curve~$X_G$ is functorial with respect to the open subgroup~$G\subseteq\GL_2(\hat\Z)$. More precisely, if~$G$ and~$G'$ are open subgroups and if~$g\in\GL_2(\hat\Z)$ is an element such that
\[
gGg^{-1}\subseteq G' \,,
\]
then the level of~$G'$ divides the level~$N$ of~$G$, and the automorphism of~$X(N)$ given by the action of~$g$ factors through a morphism
\[
X_G \to X_{G'}
\]
of~$\Q$-schemes. This morphism fits into a commuting square
\begin{center}
\begin{tikzcd}
    X_G \arrow[r]\arrow[d] & X_{G'} \arrow[d] \\
    \Spec(F_G) \arrow[r] & \Spec(F_{G'}) \,,
\end{tikzcd}
\end{center}
where the bottom arrow is induced from the field embedding $F_{G'}\hookrightarrow F_G$ sending~$\zeta_N$ to~$\zeta_N^{\det(g)^{-1}}$.

Two particular cases are of interest. Firstly, if~$g$ is the identity, so~$G\subseteq G'$, then one obtains a morphism of modular curves, and this is a morphism of schemes over~$\Spec(F_{G'})$. Secondly, if~$G'=gGg^{-1}$, then the horizontal morphisms in the above square are all isomorphisms, and we obtain
\begin{cor}\label{cor:base_change}
    Let~$G$ be an open subgroup of~$\GL_2(\hat\Z)$, let~$g\in\GL_2(\hat\Z)$, and set~$G'=gGg^{-1}$. Then~$F_{G'}=F_G$, and~$X_{G'}$ is $F_G$-isomorphic to the base-change of~$X_G$ along the automorphism of~$F_G$ sending~$\zeta_N$ to~$\zeta_N^{\det(g)^{-1}}$.
\end{cor}

If~$g\in\SL_2(\hat\Z)$, then this corollary shows that~$X_{G'}$ and~$X_G$ are $F_G$-isomorphic on the nose, but in general they will only be twists of one another.

\subsection{The holomorphic definition}

There is an alternative construction of modular curves which is much closer to their origins in the theory of modular forms, as described for example in Shimura's book \cite[Chapter~6]{shimura1971introduction}. Let~$\Gamma(N)\subseteq\SL_2(\Z)$ denote the subgroup consisting of matrices congruent to the identity modulo~$N$, which acts on the upper half plane~$\H$ from the left by fractional linear transformations. The quotient~$\Gamma(N)\backslash\H$ is a non-compact Riemann surface. The subfield of the field of meromorphic functions on~$\H$ generated by the functions
\begin{equation}\label{eq:fmn}
f_{(m,n)}(\tau)\coloneqq \frac{g_2(\tau)g_3(\tau)}{\Delta(\tau)}\wp\left(\frac{m\tau+n}{N},\tau\right)
\end{equation}
for~$m,n\in\Z$ is an extension~$\mathcal{F}_N$ of~$\Q$ of transcendence degree~$1$, whose field of constants is~$\Q(e^{2\pi i/N})$ \cite[Theorem~6.6(3)--(4)]{shimura1971introduction}, and thus corresponds to a smooth projective geometrically connected curve~$X(N)^{\Sh}$ over~$\Q(e^{2\pi i/N})$.

The functions~$f_{(m,n)}$ depend only on the values of~$m$ and~$n$ modulo~$N$, and so there is a right action of~$\GL_2(\Z/N\Z)$ on~$\mathcal{F}_N$ given by
\begin{equation}\label{eq:action_on_meromorphic_functions}
f_{(m,n)}\cdot\begin{pmatrix}a&b\\c&d\end{pmatrix} \coloneqq f_{(am+cn,bm+dn)} \,,
\end{equation}
see \cite[Theorem~6.6(2)]{shimura1971introduction}. The restriction of this action to~$\SL_2(\Z/N\Z)$ corresponds to the action of~$\SL_2(\Z/N\Z)$ on~$\Gamma(N)\backslash\H$ by fractional linear transformations, due to the identity
\[
f_{(am+cn,bm+dn)}(\tau) = f_{(m,n)}\left(\frac{a\tau+b}{c\tau+d}\right)
\]
valid for~$\begin{pmatrix}a&b\\c&d\end{pmatrix}\in\SL_2(\Z)$. If~$G$ is an open subgroup of~$\GL_2(\hat\Z)$ of level~$N\geq3$, let~$F_G^{\Sh}\subseteq\Q(e^{2\pi i/N})$ denote the abelian extension of~$\Q$ corresponding to the open subgroup $\det(G)\subseteq\hat{\Z}^\times$. According to \cite[Proposition~6.27]{shimura1971introduction}, the field of constants of the fixed field~$\mathcal{F}_N^{\text{$G$ mod $N$}}$ is~$F_G^{\Sh}$, and so~$\mathcal{F}_N^{\text{$G$ mod $N$}}$ is the fraction field of a geometrically connected curve~$X_G^{\Sh}$ over~$F_G^{\Sh}$. Scheme-theoretically, $X_G^{\Sh}$ is the quotient of~$X(N)^{\Sh}$ by the left action of~$\GL_2(\Z/N\Z)$, and so $X(N)^{\Sh}_{\C}$ is the smooth compactification of the Riemann surface $(G\cap\SL_2(\Z))\backslash\H$, where~$X(N)^{\Sh}_{\C}$ denotes the base-change of~$X(N)^{\Sh}$ along the inclusion $F_G^{\Sh}\hookrightarrow\C$.

This definition of modular curves also satisfies the same functoriality properties, see \cite[\S6.7]{shimura1971introduction}.

\subsection{Comparing the definitions}

We want to carefully relate these two constructions of modular curves. There are two points worth remarking upon. Firstly, the fields of definition~$F_G$ and~$F_G^{\Sh}$ for~$X_G$ and~$X_G^{\Sh}$ are, although isomorphic, not actually equal: $F_G$ is a subfield of~$\Q(\zeta_N)=\Q[T]/(\Phi_N(T))$, which is not a subfield of~$\C$, while~$F_G^{\Sh}$ is a subfield of~$\Q(e^{2\pi i/N})\subset\C$. There are (at least) two plausible ways that these fields could be identified, by identifying~$\zeta_N$ with either~$e^{2\pi i/N}$ or~$e^{-2\pi i/N}$. It will turn out that the former is the correct identification for us, and so from now on we permit ourselves to identify~$F_G=F_G^{\Sh}$ using the identification~$\zeta_N=e^{2\pi i/N}$.

The second point is rather more serious. Once we have identified~$X(N)$ and~$X(N)^{\Sh}$, there is no guarantee that their~$\GL_2(\Z/N\Z)$-actions will have to match up. Indeed, there are two plausible possibilities: either these actions are the same, or they differ by the automorphism of~$\GL_2(\Z/N\Z)$ given by inverse transposition. Unfortunately, it turns out to be the latter which is the case.

\begin{prop}\label{prop:two_perspectives_on_modular_curves}
    For any~$N\geq3$, there is a canonical isomorphism
    \[
    \phi\colon X(N)^{\Sh} \xrightarrow\sim X(N)
    \]
    of curves over~$\Q(\zeta_N)$ which satisfies
    \begin{equation}\label{eq:transpose-equivariance}
    \phi(\gamma\cdot x) = \gamma^{-\intercal}\cdot\phi(x)
    \end{equation}
    for all points~$x$ of~$X(N)^{\Sh}$ and all~$\gamma\in\GL_2(\Z/N\Z)$, where~$(-)^{-\intercal}$ denotes the inverse transpose.
\end{prop}

\begin{cor}
    Let~$G\subseteq\GL_2(\hat\Z)$ be an open subgroup, and let~$G^{\intercal}$ denote the transposed subgroup. Then $X_G^{\Sh}$ is canonically $F_G$-isomorphic to~$X_{G^{\intercal}}$.
\end{cor}

Because of the subtleties in the issues of the base field and $\GL_2(\Z/N\Z)$-equivariance, we give a careful proof. Let~$\cO(\H)$ denote the ring of meromorphic functions on the upper half plane~$\H$, and let~$E_{\H}$ denote the (algebraic) elliptic curve over~$\cO(\H)$ defined by the Weierstrass equation
\[
y^2 = 4x^3 - g_2(\tau)x - g_3(\tau) \,.
\]
This elliptic curve comes with a trivialisation of its $N$-torsion, given by
\[
P_{\H} = \left(\wp\left(\frac{\tau}N,\tau\right),\wp'\left(\frac{\tau}N,\tau\right)\right) \quad\text{and}\quad Q_{\H} = \left(\wp\left(\frac1N,\tau\right),\wp'\left(\frac1N,\tau\right)\right) \,.
\]
If~$\tau\in\H$ is a point in the upper half plane, we denote by~$(E_\tau,P_\tau,Q_\tau)$ the specialisation of~$(E_{\H},P_{\H},Q_{\H})$ at~$\tau$, i.e.~the base-change along the evaluate-at-$\tau$ map $\cO(\H)\to\C$. So~$(E_\tau,P_\tau,Q_\tau)$ are defined by the same formulae as~$(E_{\H},P_{\H},Q_{\H})$, except that now~$\tau$ denotes a specific complex number rather than the standard coordinate on~$\H$. The uniformisation of complex elliptic curves ensures that~$E_\tau^{\an}$ is biholomorphic to $\C/(\Z\oplus\Z\cdot\tau)$, and that~$P_\tau$ and~$Q_\tau$ are the images of~$\frac{\tau}{N}\in\C$ and~$\frac1N\in\C$, respectively.

We use this to compute the Weil pairing of~$P_{\H}$ and~$Q_{\H}$.

\begin{lemma}\label{lem:weil_vs_intersection}
    We have $e_N(P_{\H},Q_{\H})=e^{2\pi i/N}$.
\end{lemma}
\begin{proof}
    It suffices to show that~$e_N(P_\tau,Q_\tau)=e^{2\pi i/N}$ for each~$\tau\in\H$. Let~$\sigma(z)$ denote the Weierstrass sigma function relative to~$\tau$ \cite[Section~VI.3]{silverman2009arithmetic}. It is a holomorphic function on~$\C$ with simple zeroes on~$\Z\oplus\Z\tau$. According to the recipe in \cite[Proposition~VI.3.4]{silverman2009arithmetic}, the function
    \[
    g(z) \coloneqq \frac{\sigma(z)}{\sigma(z-1)}\cdot\prod_{\substack{0\leq m<N\\0\leq n<N}}\left(\frac{\sigma(z-\frac{1}{N^2}+\frac{m\tau+n}{N})}{\sigma(z+\frac{m\tau+n}{N})}\right)
    \]
    is doubly periodic, coming from a meromorphic function on~$E_\tau$ with divisor~$[N]^*Q-E[N]$. So, according to the definition in \cite[Section~III.8]{silverman2009arithmetic}, the Weil pairing of~$P_\tau$ and~$Q_\tau$ is given by
    \begin{align*}
        e_N(P_\tau,Q_\tau) &= \frac{g(z+\frac{\tau}{N})}{g(z)} \\
         &= \frac{\sigma(z+\frac{\tau}{N})\sigma(z-1)}{\sigma(z+\frac{\tau}{N}-1)\sigma(z)}\cdot\prod_{{\substack{0\leq m<N\\0\leq n<N}}}\left(\frac{\sigma(z-\frac{1}{N^2}+\frac{(m+1)\tau+n}{N})\sigma(z+\frac{m\tau+n}{N})}{\sigma(z+\frac{(m+1)\tau+n}{N})\sigma(z-\frac{1}{N^2}+\frac{m\tau+n}{N})}\right) \\
         &= \frac{\sigma(z+\frac{\tau}{N})\sigma(z-1)}{\sigma(z+\frac{\tau}{N}-1)\sigma(z)}\cdot\prod_{0\leq n<N}\left(\frac{\sigma(z-\frac{1}{N^2}+\tau+\frac{n}{N})\sigma(z+\frac{n}{N})}{\sigma(z+\tau+\frac{n}{N})\sigma(z-\frac{1}{N^2}+\frac{n}{N})}\right) \\
         &= e^{2\eta_1(z+\frac{\tau}{N}-\frac12)}\cdot e^{-2\eta_1(z-\frac12)}\cdot\prod_{0\leq n<N}\left(e^{2\eta_2(z-\frac{1}{N^2}+\frac{n}{N}+\frac{\tau}{2})}\cdot e^{-2\eta_2(z+\frac{n}{N}+\frac{\tau}{2})}\right) \\
         &= e^{2\frac{\eta_1\tau}{N}-2\frac{\eta_2}{N}} = e^{2\pi i/N} \,,
    \end{align*}
    where~$2\eta_1$ and~$2\eta_2$ are the quasiperiods. Here, we use the transformation laws
    \[
        \frac{\sigma(z+1)}{\sigma(z)} = -e^{2\eta_1(z+\frac12)} \quad\text{and}\quad \frac{\sigma(z+\tau)}{\sigma(z)} = -e^{2\eta_2(z+\frac{\tau}{2})}
    \]
    for~$\sigma$ in the fourth line \cite[Theorem~IV.3]{chandrasekharan}
    , and use Legendre's relation $\eta_1\tau-\eta_2=\pi i$ in the final line \cite[Theorem~IV.2]{chandrasekharan}. Comparing our values for the Weil pairing and the intersection pairing, the result follows.
\end{proof}

\begin{rmk}
    The basis $(P_\tau,Q_\tau)$ of~$E_\tau[N]$ is \emph{negatively} oriented, in the sense that the intersection product~$P_\tau\cdot Q_\tau=-1$ when we identify~$E_\tau[N]\cong H_1(E_\tau,\Z/N\Z)$ in the usual way \cite[Proposition~VI.5.6(b)]{silverman2009arithmetic}. So Lemma~\ref{lem:weil_vs_intersection} implies that the Weil pairing and intersection pairing on a complex elliptic curve are related by
    \[
    e_N(P,Q) = e^{-2\pi iP\cdot Q/N} \,.
    \]
\end{rmk}

As a consequence of Lemma~\ref{lem:weil_vs_intersection}, the triple~$(E_{\H},P_{\H},Q_{\H})$ defines a $\cO(\H)$-valued point of~$Y(N)$, viewing both as~$\Q(\zeta_N)$-schemes and using Lemma~\ref{lem:new_moduli_description}. If we write~$Y(N)$ as the zero-set of some polynomials over~$\Q(\zeta_N)$, then this $\cO(\H)$-valued point is given by a tuple of holomorphic functions satisfying the defining equations of~$Y(N)$, and so corresponds to a holomorphic map
\[
\H \to Y(N)_{\C}^{\an} \,.
\]
($Y(N)_{\C}$ denotes the base-change of~$Y(N)$ along~$\Q(\zeta_N)\hookrightarrow\C$.) On the level of sets, the above map takes~$\tau\in\H$ to the class of the complex elliptic curve~$(E_\tau,P_\tau,Q_\tau)$. Two triples~$(E_\tau,P_\tau,Q_\tau)$ and~$(E_{\tau'},P_{\tau'},Q_{\tau'})$ are isomorphic if~$\tau$ and~$\tau'$ lie in the same $\Gamma(N)$-orbit, so this map factors through a holomorphic map
\[
\phi\colon Y(N)_{\C}^{\Sh,\an}  = \Gamma(N)\backslash\H \to Y(N)_{\C}^{\an} \,.
\]
Any complex elliptic curve~$E$ with a basis~$(P,Q)$ of~$E[N]$ satisfying~$e_N(P,Q)=e^{2\pi i/N}$ is isomorphic to~$(E_\tau,P_\tau,Q_\tau)$ for some value of~$\tau\in\H$ which is unique up to the action of~$\Gamma(N)$, and so~$\phi$ is bijective. Since~$Y(N)$ is reduced and smooth, this implies that~$\phi$ is a biholomorphism, and so it is the analytification of an isomorphism
\[
\phi\colon Y(N)_{\C}^{\Sh} \xrightarrow\sim Y(N)_{\C}
\]
of complex algebraic curves.

Next, we show that~$\phi$ is the base change of an isomorphism of curves over~$\Q(\zeta_N)$. For this, we construct elements of the function field of~$Y(N)$ as follows. Fix~$m,n\in\Z$ and suppose that~$E$ is an elliptic curve over a $\Q$-algebra~$R$ which is globally given by a Weierstrass equation
\[
y^2 = 4x^3 - c_2x - c_3 \,,
\]
where~$c_2,c_3\in R$ with $\Delta=c_2^3-27c_3^2$ a unit. If~$(P,Q)$ is a basis of~$E[N]$ defined over~$R$, then the $x$-coordinate~$x(mP+nQ)$ is again an element of~$R$. The element
\[
h_{(m,n)}(E,P,Q)\coloneqq\frac{c_2c_3}{\Delta}\cdot x(mP+nQ) \in R
\]
does not depend on the choice of equation (since the equation is unique up to a scaling $(x,y)\mapsto(u^2x,u^3y)$). If we apply this construction to the universal elliptic curve~$E_N$ over~$Y(N)$ and its tautological trivialisation~$(P_N,Q_N)$, then we may cover~$Y(N)$ with affine opens~$U_i$ over which~$E_N$ admits a global Weierstrass equation, and then the resulting elements $$h_{(m,n)}(E_N|_{U_i},P_N|_{U_i},Q_N|_{U_i})$$ glue together to elements
\[
h_{(m,n)}\in\cO(Y(N)) \,.
\]

By construction, $\phi^*h_{(m,n)}\in\cO(\H)$ is the meromorphic function $f_{(m,n)}\in\mathcal{F}_N$ defined earlier~\eqref{eq:fmn}. So the function field of~$Y(N)^{\Sh}$ is contained in the function field of~$Y(N)$ when we view both as subfields of the field of meromorphic functions on~$Y(N)^{\Sh}_{\C}\cong Y(N)_{\C}$. This means that the inverse of~$\phi$ is the base-change of a morphism of curves defined over~$\Q(\zeta_N)$: this morphism is then automatically an isomorphism, and so~$\phi$ is also defined over~$\Q(\zeta_N)$.

This completes the proof of Proposition~\ref{prop:two_perspectives_on_modular_curves}, except for the assertion regarding the $\GL_2(\Z/N\Z)$-actions. For this, the left action of~$\GL_2(\Z/N\Z)$ on~$Y(N)$ induces a right action on~$\cO(Y(N))$, which on the elements~$h_{(m,n)}$ is given by
\[
    h_{(m,n)}\cdot\begin{pmatrix}a&b\\c&d\end{pmatrix} = \frac{c_2c_3}{\Delta}\cdot x(m(a'P+c'Q)+n(b'P+d'Q)) = h_{(a'm+b'n,c'm+d'n)} \,,
\]
where~$a',b',c',d'$ are as in~\eqref{eq:inverse_matrix}. Comparing with the action on~$f_{(m,n)}$ from~\eqref{eq:action_on_meromorphic_functions}, we see that the actions of~$\GL_2(\Z/N\Z)$ on the function fields of~$Y(N)^{\Sh}$ and~$Y(N)$ differ by inverse transposition, which is what we wanted to show. \qed

\begin{rmk}
    \leavevmode
    \begin{itemize}
        \item The curves~$X_G$ and~$X_G^{\Sh}=X_{G^{\intercal}}$ can genuinely be non-isomorphic over the base field~$F_G$, see~\cite[Remark~2.2]{rszbadic}. On the other hand, they are always \emph{geometrically} isomorphic, i.e.~twists of one another. This is because any element of~$\SL_2$ is conjugate to its inverse transpose via the matrix~$\begin{pmatrix}0&1\\-1&0\end{pmatrix}$, and this means that~$X_{G,\C}^{\Sh,\an} = (G\cap\SL_2(\Z))\backslash\H$ and~$X_{G^{\intercal},\C}^{\Sh,\an} = (G^{\intercal}\cap\SL_2(\Z))\backslash\H$ are biholomorphic.
        \item For the subgroup~$H$ considered in Section~\ref{sec3}, the curves~$X_H$ and~$X_H^{\Sh}$ are isomorphic, because~$H$ and~$H^{\intercal}$ are conjugate via~$\begin{pmatrix}0&1\\-1&0\end{pmatrix}\in\SL_2(\Z/27\Z)$. So it does not matter which definition we use in Section~\ref{sec3}; we are talking about the same modular curve.
    \end{itemize}
\end{rmk}


\printbibliography
\end{document}

%% file: abbreviations_xns.tex
\usepackage{geometry} 
\usepackage{amsfonts}
\usepackage{amsmath}
\usepackage{mathrsfs}
\usepackage{amssymb}
\usepackage{amsthm}
\usepackage{tikz}
\usepackage{tikz-cd}
\usepackage{listings}
\usepackage{color}
\usepackage[shortlabels]{enumitem}
\usepackage{imakeidx}
\usepackage[colorinlistoftodos]{todonotes}
\usepackage{upgreek}
\hypersetup{colorlinks=true, linkcolor=teal, , urlcolor=blue}
\usepackage{mdframed}
\usepackage{manfnt}
\usepackage[new]{old-arrows}
\usepackage{stmaryrd}
\usepackage{appendix}
\usepackage[normalem]{ulem}
\usepackage{ytableau}
\usepackage{blkarray}
\usepackage{mathtools}
\usepackage{verbatim}
\usepackage{graphicx}
\usepackage{float}

\DeclareMathAlphabet{\matheuler}{U}{eus}{m}{n}

\makeindex[intoc]

\definecolor{codegreen}{rgb}{0,0.6,0}
\definecolor{codegray}{rgb}{0.5,0.5,0.5}
\definecolor{codepurple}{rgb}{0.58,0,0.82}
\definecolor{backcolour}{rgb}{0.95,0.95,0.92}

\newcommand{\MAT}[7]{\begin{pmatrix}{#1}&{#2}&{#3}&{#4}\\{#5}&{#6}&{#7}&\MATHelper}
\newcommand{\MATHelper}[9]{{#1}\\{#2}&{#3}&{#4}&{#5}\\{#6}&{#7}&{#8}&{#9}\end{pmatrix}}

\newcommand{\Frob}{\mrm{Frob}}

\DeclareFontFamily{U}{wncy}{}
\DeclareFontShape{U}{wncy}{m}{n}{<->wncyr10}{}
\DeclareSymbolFont{mcy}{U}{wncy}{m}{n}
\DeclareMathSymbol{\Sha}{\mathord}{mcy}{"58} 

\makeatletter
\newcommand{\bigperp}{%
  \mathop{\mathpalette\bigp@rp\relax}%
  \displaylimits
}

\newcommand{\bigp@rp}[2]{%
  \vcenter{
    \m@th\hbox{\scalebox{\ifx#1\displaystyle2.1\else1.5\fi}{$#1\perp$}}
  }%
}
\makeatother

\DeclareMathOperator{\Lie}{Lie}

\DeclareMathOperator{\Spec}{Spec}

\DeclareMathOperator{\Sh}{Sh}

\let\templim\lim
\renewcommand{\lim}{\templim\limits}

\newcommand{\vecc}[1]{\overrightarrow{#1}}

\newcommand\homses[4]{
    \begin{tikzcd}[ampersand replacement=\&]
        0\ar[r]\&{\displaystyle #1}\ar[r, "{#2}"]\ar[r]\&{\displaystyle #3}\ar[r, "{#4}"]\&
        \homsesHELPER
}
\newcommand\homsesHELPER[9]{
        {\displaystyle #1}\ar[r]\& 0\\
        0\ar[r]\&{\displaystyle #5}\ar[r, "{#6}"]\&{\displaystyle #7}\ar[r, "{#8}"]\&{\displaystyle #9}\ar[r]\& 0
        \ar[from=lllu, to=lll, "{#2}"]\ar[from=llu, to=ll, "{#3}"]\ar[from=lu, to=l, "{#4}"]
    \end{tikzcd}
}

\newcommand{\mf}{\mathfrak}

\newcommand{\mfq}{\mf q}

\newcommand{\mc}{\mathcal} 
\newcommand{\mcA}{\mc A}

\newcommand{\mcO}{\mc O}

\newcommand{\mcE}{\mc E}

\newcommand{\mcM}{\mc M}

\newcommand{\mrm}{\mathrm}

\newcommand{\A}{\mathbf A}

\renewcommand{\H}{\mathbb H}

\renewcommand{\tau}{\uptau}
\renewcommand{\P}{\mathbf P}

\newcommand{\dparens}[1]{\!\left(\!\left(#1\right)\!\right)}
\newcommand\dps\dparens 

\newcommand{\too}{\longrightarrow}

\makeatletter
\newcommand*{\da@rightarrow}{\mathchar"0\hexnumber@\symAMSa 4B }
\newcommand*{\da@leftarrow}{\mathchar"0\hexnumber@\symAMSa 4C }
\newcommand*{\xdashrightarrow}[2][]{%
  \mathrel{%
    \mathpalette{\da@xarrow{#1}{#2}{}\da@rightarrow{\,}{}}{}%
  }%
}
\newcommand{\xdashleftarrow}[2][]{%
  \mathrel{%
    \mathpalette{\da@xarrow{#1}{#2}\da@leftarrow{}{}{\,}}{}%
  }%
}
\newcommand*{\da@xarrow}[7]{%
  \sbox0{$\ifx#7\scriptstyle\scriptscriptstyle\else\scriptstyle\fi#5#1#6\m@th$}%
  \sbox2{$\ifx#7\scriptstyle\scriptscriptstyle\else\scriptstyle\fi#5#2#6\m@th$}%
  \sbox4{$#7\dabar@\m@th$}%
  \dimen@=\wd0 %
  \ifdim\wd2 >\dimen@
    \dimen@=\wd2 %
  \fi
  \count@=2 %
  \def\da@bars{\dabar@\dabar@}%
  \@whiledim\count@\wd4<\dimen@\do{%
    \advance\count@\@ne
    \expandafter\def\expandafter\da@bars\expandafter{%
      \da@bars
      \dabar@ 
    }%
  }%
  \mathrel{#3}%
  \mathrel{%
    \mathop{\da@bars}\limits
    \ifx\\#1\\%
    \else
      _{\copy0}%
    \fi
    \ifx\\#2\\%
    \else
      ^{\copy2}%
    \fi
  }%
  \mathrel{#4}%
}
\makeatother
\newcommand\xdashto\xdashrightarrow
\newcommand\xdashfrom\xdashleftarrow

\makeatletter
\providecommand{\leftsquigarrow}{%
  \mathrel{\mathpalette\reflect@squig\relax}%
}
\newcommand{\reflect@squig}[2]{%
  \reflectbox{$\m@th#1\rightsquigarrow$}%
}
\makeatother

\newcommand*{\rom}[1]{\textup{\uppercase\expandafter{\romannumeral#1}}}

\renewcommand{\bar}{\overline}

\renewcommand\t\text 
\newcommand\ttt\texttt

\newcommand{\important}[1]{\textit{#1}}
\newcommand\noteworthy\important

\pgfmathdeclarerandomlist{thuses}{{Thus}{Therefore}{As such}{As we wished}{By this great fortune}{Alors}{Donc}{Following this thread}{Consequently}{Thusforth}{Rejoice}{This can only mean one thing...}{Ergo}{Voila}}

\pgfmathdeclarerandomlist{fibers}{{fiber}{fibre}}

\pgfmathdeclarerandomlist{wewins}{{we win}{it's over}{all is good}{victory has been achieved}{we've crossed the finish line}{we're happy campers}{we've done it}{that's a wrap}{we have the high ground}}

\tikzset{%
	symbol/.style={%
		,draw=none
		,every to/.append style={%
			edge node={node [sloped, allow upside down, auto=false]{$#1$}}}
	}
}

\makeatletter
\def\renewtheorem#1{%
	\expandafter\let\csname#1\endcsname\relax
	\expandafter\let\csname c@#1\endcsname\relax
	\gdef\renewtheorem@envname{#1}
	\renewtheorem@secpar
}
\def\renewtheorem@secpar{\@ifnextchar[{\renewtheorem@numberedlike}{\renewtheorem@nonumberedlike}}
\def\renewtheorem@numberedlike[#1]#2{\newtheorem{\renewtheorem@envname}[#1]{#2}}
\def\renewtheorem@nonumberedlike#1{  
	\def\renewtheorem@caption{#1}
	\edef\renewtheorem@nowithin{\noexpand\newtheorem{\renewtheorem@envname}{\renewtheorem@caption}}
	\renewtheorem@thirdpar
}
\def\renewtheorem@thirdpar{\@ifnextchar[{\renewtheorem@within}{\renewtheorem@nowithin}}
\def\renewtheorem@within[#1]{\renewtheorem@nowithin[#1]}
\makeatother

\DeclareDocumentEnvironment{MyFrame}{O{1cm}O{0.4pt}O{0.8cm}O{black}O{3}O{2ex}}
{\par\hfill\rlap{%
		\bgroup\color{#4}%
		\hskip-\dimexpr#1-#3\relax\rule{#1}{#2}%
		\hskip-\dimexpr#1/#5\relax\rule[-\dimexpr#1-\dimexpr#1/#5\relax]{#2}{#1}%
		\egroup
	}%
	\vskip-\dimexpr#1/#5+\dimexpr#1/#5-#6\relax%
}
{\par\nobreak\offinterlineskip\vskip-\dimexpr#1/#5+\dimexpr#1/#5-#6\relax\noindent%
	\hskip-#3\bgroup\color{#4}%
	\rule{#1}{#2}\hskip-\dimexpr#1-\dimexpr#1/#5-#2\relax%
	\rule[-\dimexpr#1/#5-#2\relax]{#2}{#1}\egroup\par
}

\newtheorem{innercustomgeneric}{\customgenericname}
\providecommand{\customgenericname}{}
\newcommand{\newcustomtheorem}[2]{%
	\newenvironment{#1}[1]
	{%
		\renewcommand\customgenericname{#2}%
		\renewcommand\theinnercustomgeneric{##1}%
		\innercustomgeneric
	}
	{\endinnercustomgeneric}
}
\newtheorem{thm}{Theorem}[section]

\newtheorem{cor}[thm]{Corollary}

\newtheoremstyle{break}
  {\topsep}{\topsep}%
  {\itshape}{}%
  {\bfseries}{}%
  {\newline}{}%
\theoremstyle{break}

\theoremstyle{break}

\renewtheorem{conj}[thm]{Conjecture}

\newcustomtheorem{Prob}{Problem}
\newcustomtheorem{Exc}{Theoretical Exercise}

\theoremstyle{definition}
\newtheorem{defninner}[thm]{Definition}

\newtheorem{recinner}[thm]{Recall}

\newtheorem{Assump}{Assumption}

\newtheorem*{nonexinner}{Non-example}
\newtheorem{warninner}[thm]{Warning}

\newtheorem*{ansinner}{Answer}

\newtheorem*{histinner}{History}

\theoremstyle{remark}
\newtheorem{rmk}[thm]{Remark}

\newcommand\remsymbol{$\circ$}
\newcommand\anssymbol{$\star$}
\newcommand\warnsymbol{$\bullet$}
\newcommand\proofsymbol{$\blacksquare$}
\newcommand\nonexsymbol{$\triangledown$}
\newcommand\recsymbol{$\odot$}
\newcommand\defnsymbol{$\diamond$}
\newcommand\histsymbol{$\ominus$}

\let\proofinner=\proof

\newcommand\myproof[1][Proof]{\proofinner[#1]\renewcommand\qedsymbol\proofsymbol}
\def\proof{\myproof}
